\documentclass[reqno, 10pt]{amsart}

\usepackage{amssymb,  latexsym}
\usepackage{hyperref}
\usepackage{amsmath}
\usepackage{amsthm}
\usepackage{enumerate}
\usepackage{amsfonts}
\usepackage{graphicx}
\usepackage{mathrsfs}
\usepackage{color}
\usepackage{array}
\usepackage{tabularx}
\usepackage{todonotes}
\usepackage{mathtools}

\theoremstyle{plain}
\newtheorem{theorem}{Theorem}
\newtheorem{proposition}{Proposition}[section]
\newtheorem{lemma}{Lemma}[section]

\theoremstyle{definition}
\newtheorem{definition}{Definition}[section]
\newtheorem{remark}{Remark}[section]

\newcommand{\dd}{\mathop{}\!\mathrm{d}}
\numberwithin{equation}{section}
\numberwithin{theorem}{section}

\numberwithin{equation}{section}

\newcommand{\Z}{{\mathbb{Z}}}
\newcommand{\C}{{\mathbb{C}}}
\newcommand{\R}{{\mathbb{R}}}
\newcommand{\N}{{\mathbb{N}}}
\newcommand{\T}{{\mathbb{T}}}

\renewcommand{\H}{\mathcal{H}_{A}}

\let\Re=\undefined\DeclareMathOperator*{\Re}{Re}

\def\leq{\leqslant}
\def\geq{\geqslant}

\begin{document}

\title[Growth of sobolev norms for 2d nls with partial harmonic potential]{Growth of sobolev norms for 2d cubic nls with partial harmonic potential}

\author[M. Deng]{Mingming Deng}
\address{Graduate School of China Academy of Engineering Physics,  \ Beijing, \ 100088, \ China,  }
\email{dengmingming18@gscaep.ac.cn}	
\author[X. Su]{Xiaoyan Su}
\address{Laboratory of Mathematics and Complex Systems (Ministry of Equation), School of Mathematics Sciences, Beijing Normal University,  \ Beijing, \ 100875, \ China, }
\email{suxiaoyan0427@qq.com}
\author[J. Zheng]{Jiqiang Zheng}
\address{Institute of Applied Physics and Computational Mathematics,  \ Beijing, \ 100088, \ China,  }
\email{zheng\_jiqiang@iapcm.ac.cn}		


\maketitle	
\begin{abstract}
In this paper, we study the $2$D cubic nonlinear Schr\"odinger equation (NLS) with the partial harmonic potential. First, we prove the local well-posedness in Bourgain spaces by establishing a key bilinear estimate associated with the partial harmonic oscillator.  Then, we give the polynomial bound of the Sobolev norms for the solutions using the method of the Planchon, Tzvetkov, and Visciglia in \cite{PTV21}.
\end{abstract}

\section{introduction}
We consider the following cubic nonlinear Schr\"odinger equation with the partial harmonic potential in $\mathbb R^2$:
\begin{equation}
\begin{cases}\label{equation}
i \partial_t u +Au \pm |u|^{2}u = 0,\quad (t,x,y) \in \R^3, \\
u(0,x,y) = \varphi(x,y) \in \H^s(\R^2),
\end{cases}
\end{equation}
where $A=-\partial_x^2-\partial_y^2 +y^2$ and $\H^s(\R^2)$ is the Sobolev space adapted to the operator $A$. This operator  models magnetic traps in the Bose--Einstein condensation \cite{ACS, Josser} and has been extensively studied in \cite{ACS, AC20,SWX22, SWX22'}.  The partial harmonic oscillator has an intermediate behavior between the Laplacian operator which has a continuous spectrum, and the (full) harmonic oscillator which has discrete spectrum, and shares the properties of both operators. For example,  for the linear evolution operator $e^{itA}$, the $x$-part can yield large time dispersion, but the $y$-part only provide local time dispersion due to the potential confinement. This operator also serves as an example which has infinite many embedded eigenvalues.


In this paper, we focus on the growth of Sobolev norms  for the solutions to the  nonlinear equation \eqref{equation}. This problem was first investigated by Bourgain for the Laplacian on $\mathbb T^2$ in \cite{B96}, where he used the Fourier multiplier method to prove an iteration bound
\begin{equation}\label{eq-an improved iteration bound}
\|u(t)\|_{H^s(\T^2)} \leq \|u(\tau)\|_{H^s(\T^2)} + C\|u(\tau)\|_{H^s(\T^2)}^{1-\delta}, \quad \delta^{-1} = (s-1)+
\end{equation}
for any $t\in[\tau, \tau+T]$ with $T=T(\|u_0\|_{H^1})$. Here and in the sequel, we write $a+$ ($a-$) to denote any number larger (smaller) than $a$. This implies the following polynomial bound
\begin{align*}
    \|u(t)\|_{H^s(\mathbb T^2)} \leq C|t|^{2(s-1)+},
\end{align*}
and improved the exponential bound derived by a pacing argument. The growth of Sobolev norms reflected the weak turbulence phenomena for the nonlinear dispersive equations.

For the cubic NLS on $\mathbb R^2$, a similar result has been derived by Staffilani in \cite{S97}. She showed the bound \eqref{eq-an improved iteration bound} by using bilinear estimates in Bourgain spaces and derived
    \begin{align*}
    \|u(t)\|_{H^s(\mathbb R^2)} \leq C|t|^{(s-1)+}.
\end{align*}
The growth of Sobolev norms also have been investigated for the NLS on compact Riemannian manifold, see \cite{P12} for more information.

One natural question to ask is that whether we can get similar results for the NLS with a potential.
One recent result on this field was given by Planchon--Tzvetkov--Visciglia in \cite{PTV21} where they discussed the Schr\"odinger operator with the (full) harmonic potential, and obtained polynomial bounds on the growth of solutions to cubic NLS in $\mathbb R^2$.
Their results highly depended on the following bilinear estimate established in \cite{P12}:
\begin{equation}\label{eq-bilinear estimate for tildeA}
\|P_N(u)P_M(v)\|_{L^2((0,T);L^2(\R^2))} \leq C(\min{(M,N)})^{0+}\Big(\frac{\min{(M,N)}}{\max{(M,N)}}\Big)^{\frac12-}
\|P_N(u)\|_{X^{0,b}_T(\R^2)}\|P_M(v)\|_{X^{0,b}_T(\R^2)},
\end{equation}
where $b<1/2$, $P_N,P_M$ are the Littlewood--Paley localization associated with the Hermite operator, $N,M$ are dyadic integers, and $X_T^{s, b}$ denotes the (local) Bourgain spaces adapted to the Hermite operator. Using this, they proved the growth of Sobolev norms of order $s$ with $s\geq 1$ is
\begin{align*}
    \|u(t)\|_{H^s(\mathbb R^2)}+\|\langle x\rangle^s u\|_{L^2(\mathbb R^2)} \leq Ct^{\frac{2}{3}(s-1)+}.
\end{align*}


Inspired by the work of Planchon--Tzvetkov--Visciglia \cite{P12} and the work of Staffilani \cite{S97}, it is natural to conjecture that similar results should be valid for the Schr\"odinger equation with the partial harmonic potential $A$. The aim of this paper is to confirm this conjecture and give the growth rate.

The key step to gain the growth rate is to establish the following bilinear estimate.
\begin{proposition}\label{proposition-Poiret2.3.13}
Let $\delta_0\in(0,\frac12]$ and assume  that $1 \leq M \leq N$ are dyadic integers. Then there exists  $0<b<\frac12$ and $C>0$ such that for every $\delta\in(0,\delta_0]$
\begin{align}\label{prop-1.1}
\big\| \Delta_{N}u\Delta_{M}v \big\|_{L^2(\R;L^2(\mathbb R^2))} \leq CM^\delta\Big(\frac{M}{N}\Big)^{\frac12-\delta}\big\| \Delta_{N}u \big\|_{X^{0,b}(\R\times\R^2)}\big\| \Delta_{M}v \big\|_{X^{0,b}(\R\times\R^2)},
\end{align}
where $\Delta_N$ is given in Subsection \ref{lwd} and $X^{s, b}(\R\times\R^2)$ is the Bourgain space defined in Section \ref{subsection of Bourgain space}.
\end{proposition}
To get the above result, we first give some Bernstein type inequalities associated to the Schr\"odinger operator $A$. Then, using a Morawetz type estimate and some interpolation inequalities, we obtain the  bilinear estimate  \eqref{prop-1.1} for the operator $A$.  We refer the details to Section \ref{section3}.

Proposition \ref{proposition-Poiret2.3.13} and a trilinear estimate  yield the local well-posedness result in Bourgain spaces. Then, by defining suitable modified energies $\mathcal{S}_{2k+2}$ and $\mathcal{R}_{2k+2}$, we prove our main result:
\begin{theorem}[Growth of Sobolev norms]\label{Theorem-growth}
Let $\varepsilon>0$ and $k\in\N$. For every global solution $u$ to \eqref{equation} such that $u \in\mathcal{C}(\R, \H^{2k}(\R ^2))$ and
\begin{equation}\label{eq-A1}
\sup_{t\in\R}\|u(t, x,y)\|_{\H^1(\R ^2)}<\infty,
\end{equation}
there exists a constant $C$ such that
\begin{equation}\label{eq-growth}
\| u \|_{\H^{2k}(\R ^2)} \leq C\langle t \rangle^{\frac{2}{3}(2k-1)+\varepsilon}.
\end{equation}
\end{theorem}

\begin{remark}
  By conservation of energy, the assumption \eqref{eq-A1} is naturally satisfied for the defocusing case, but not for the focusing case.
\end{remark}


The organization of the paper is as follows. In Section \ref{section2}, we introduce the preliminaries which includes the definition of Hermite functions, some function spaces adapted to the operator $A$, and Bernstein-type inequalities and a multiplier theorem.  In Section \ref{section3}, we prove the Proposition 1.1 and  the local well-posedness for the equation \eqref{equation} using Bourgain spaces;  In Section \ref{sec-growth}, we obtain the growth of high order Sobolev norms for solution to \eqref{equation}.
\section{preliminaries}
\label{section2}
Let $z=(x,y)\in \R^2 $ be the space variable, and $\zeta=(\xi, \eta)\in \R^2$ be the frequency variable. Denote $A_x=-\partial_x^2$ and $A_y=-\partial_y^2+y^2$.

\subsection{Hermite functions (See \cite{Than93book})}

The Hermite functions $h_k$ are defined on $\mathbb{R}$ by the formula
\begin{align*}
h_k(y) = (2^kk!\pi^\frac12)^{-\frac12}H_k(y)e^{-\frac{y^2}{2}},\quad k\in\N,
\end{align*}
where $H_k$ denotes the Hermite polynomial of degree $k$.
It is the eigenfunction of $A_y$ with eigenvalues $2k+1$, that is,
$A_yh_k = (2k+1)h_k.$
In addition, the functions $\{h_k\}_{k\in \mathbb N}$  form an orthonormal family in $L^2(\R)$. Hence,  any function $f \in L^2(\mathbb R)$, has an expansion in Hermite functions
$$f(y) = \sum_{k=0}^\infty P_kf(y),$$
where $P_k$ is the projection to the span of $h_k$, defined by
\begin{align*}
P_kf(x) = \int_{\R}f(y)h_k(y)\dd y h_k(x).
\end{align*}

The next lemma gives the almost orthogonality for the  product of Hermite functions with different frequencies.
\begin{lemma}[See \cite{BTT13}]\label{Lemma-BTT-lemma 7.5}
Let $\delta>0$ and $K\geq0$. Assume $N_0\geq N_1^{1+\delta}$ and $N_1\geq N_2\geq N_3$, there exists $C_K>0$ such that for every $k_i$ satisfying $N_i\leq\sqrt{2k_i+1}\leq2N_i$, $ i=0,1,2,3$, we have
\begin{align*}
\Big| \int_\R h_{k_0}h_{k_1}h_{k_2}h_{k_3}\dd y \Big| \leq C_K(1+N_0)^{-K}.
\end{align*}
\end{lemma}

\subsection{Sobolev spaces adapted to $A$ (See \cite{SWX22'})}
Using the Fourier transform with respect to $x$ and the Hermite expansion with respect to $y$,  the operator $A$ when acting on a function $u\in C_0^\infty(\mathbb R^2)$ can be written as
\begin{align*}
Au(x,y) = \sum_{k=0}^\infty\int_\R e^{ix\xi} (\xi^2 + 2k + 1)\mathcal{F}_{x \to \xi}P_ku(\xi,y)\dd \xi.
\end{align*}
By the functional calculus, we define the fractional power $A^{s/2}$ for $s\in \mathbb R$ on $ C_0^\infty(\mathbb R^2)$ by
\begin{align*}
    A^{s/2}u(x,y) = \sum_{k=0}^\infty \int_\R e^{ix\xi}(\xi^2 + 2k + 1)^{s/2}\mathcal{F}_{x\to \xi}P_ku(\xi,y)\dd \xi.
\end{align*}
 The Sobolev space $\H^s(\R^2)$  adapted to $A$ with index $s\in \mathbb R$ is defined as
\begin{align*}
\H^s(\R^2) = D(A^{s/2})=\big\{ u\in L^2(\R^2): A^{s/2}u\in L^2(\R^2) \big\}
\end{align*}
with the norm $\|u\|_{\H^s(\R^2)} = \big\|A^{s/2}u\big\|_{L^2(\R^2)}$.
\begin{remark}

From \cite{DG09}, we have the following equivalence of the Sobolev norms:
for every $s\geq0$, there exists a constant $C>0$ such that
\begin{equation}\label{eq-equivalence}
\frac1C\big( \big\| D^s u \big\|_{L^2(\R ^2)}^2 + \big\| \langle y \rangle^su \big\|_{L^2(\R ^2)}^2 \big) \leq \|u\|_{\H^s(\R ^2)}^2 \leq
C\big( \big\| D^s u \big\|_{L^2(\R ^2)}^2 + \big\| \langle y \rangle^su \big\|_{L^2(\R ^2)}^2\big),
\end{equation}
where $D^s$ is the operator associated with the Fourier multiplier $|\xi^2+\eta^2|^{s/2}$ and $\langle y \rangle = \big(1+y^2\big)^\frac12$.
Using this result, obtaining  \eqref{eq-growth} is reduced to the estimate
\begin{equation}
\big\| D^{2k}u \big\|_{L^2(\R ^2)} + \big\| \langle y \rangle^{2k} u \big\|_{L^2(\R ^2)} \leq C\langle t \rangle^{\frac{2k-1}{2}+\varepsilon}.
\end{equation}
\end{remark}
For later use, we recall a Mikhlin-type multiplier theorem for the operator $A$ in \cite{SWX22}. Let $\mathcal D^N_k$ be the forward finite difference with respect to the variable $k$ of order $N$.
\begin{lemma}
[Multiplier theorem for $A$]\label{Theorem-mikhlin multiplier theorem}
Suppose that a function $m(\xi,k)$ defined on $\R\times\N$ satisfies that for all $0 \leq N \leq 2$,
\begin{align*}
\Big| \frac{\partial^N}{\partial {\xi}^N}m(\xi,k) \Big| \leq C(\xi^2+2k+1)^{-\frac N2}\quad \text{and}\quad\big|\mathcal D^N_k m(\xi,k) \big|\leq C(\xi^2+2k+1)^{-N}.
\end{align*}
Then, the operator $T_m$ defined by
\begin{align*}
T_mf(x,y) = \sum_{k=0}^\infty\int_\R e^{ix\xi}m(\xi,k)\mathcal{F}_{x \to \xi}P_kf(\xi,y)\dd \xi
\end{align*}
is bounded on $L^p(\R^2)$ for any $1<p<\infty$.
\end{lemma}

%

\subsection[Littlewood--Paley theory]{Littlewood--Paley theory for \(A\)}\label{lwd}
Let $\phi$ be a smooth function defined on $\mathbb R^+$ such that $\phi(\lambda) = 1$ for $0\leq\lambda\leq1$, and $\phi(\lambda)=0$ for $\lambda\geq2$. For each dyadic integer $N\in2^{\Z}$, we define
\begin{align*}
\phi_N(\lambda) := \phi\Big(\frac{\lambda}{N^2}\Big), \quad \text{and} \quad \psi_N(\lambda) := \phi_N(\lambda) - \phi_{\frac N2}(\lambda).
\end{align*}
Clearly, $\{ \psi_N(\lambda) \}_{N\in2^{\Z}}$ forms a partition of unity for $\lambda\in(0,\infty)$. Then we define the Littlewood-Paley projection as follows:
\begin{align*}
S_Nu(z) &:=  \phi_N( A )u(z) = \sum_{k=0}^\infty\int_\R e^{ix\xi}\phi\Big( \frac{\xi^2+2k+1}{N^2} \Big)\mathcal{F}_{x \to \xi}P_ku(\xi,y)\dd \xi,\\
\Delta_Nu(z) &:=  \psi_N (A)u(z) = \sum_{k=0}^\infty\int_\R e^{ix\xi}\psi\Big( \frac{\xi^2+2k+1}{N^2} \Big)\mathcal{F}_{x \to \xi}P_ku(\xi,y)\dd \xi.
\end{align*}

  For the frequency localized operator $S_N$  and $\Delta_N$, we have the following estimates.
\begin{lemma}[Bernstein-type estimates]\label{Lemma-Bernstein estimate}
Let $1<p\leq q<\infty$ and $s\geq 0$.
\begin{enumerate}
\item[$(1)$] $\|A^{s/2}S_Nu\|_{L^q(\R^2)} \lesssim_s N^{s+(\frac2p-\frac2q)}\|u\|_{L^p(\R^2)}$ \label{Bernstein-1};
\item[$(2)$] $ N^{s}\|\Delta_N u\|_{L^p(\R^2)}\lesssim_s \|A^{s/2}\Delta_Nu\|_{L^p(\R^2)} \lesssim_s N^{s}\|\Delta_N u\|_{L^p(\R^2)}$;
\end{enumerate}
\end{lemma}
\begin{proof}
   Let $e^{-tA}$ be the heat semigroup associated to the operator $A$.
We prove (1) as follows. Using Lemma \ref{Theorem-mikhlin multiplier theorem}, we have for $1<p<\infty$ that
\begin{align*}
\Big\|A^{s/2} \phi\Big(\frac{A}{N^2}\Big)e^{\frac{A}{N^2}}e^{-\frac{A}{N^2}}f \Big\|_{L^q(\R^2)} \leq C N^s \|\lambda^{s/2} \phi(\lambda) e^{\lambda}\|_{C^2(0,2)} \|e^{-\frac{A}{N^2}}f \|_{L^q(\R^2)}.
\end{align*}
By Mehler's formula, the Schwartz kernel of $e^{-tA}$ is given by
\begin{align*}
K(t,z,z') = 2^{-\frac32}\pi^{-1}t^{-\frac12}(\sinh 2t)^{-\frac12}e^{-\frac{(x - x')^2}{4t} + \frac14(2\coth 2t - \tanh t)(y - y')^2 + \frac14\tanh t(y + y')^2}.
\end{align*}
Since $\sinh(2N^{-2})\leq C2N^{-2}$, $\coth(2N^{-2})\leq C\frac12N^2$, we have
\begin{align*}
\begin{aligned}
\big|K(N^{-2},z,z') \big| \leq  & ~C N^2 e^{-\big( \frac14N^2(x - x')^2 + \frac12\coth(2N^{-2})(y - y')^2 + \frac14\tanh(N^{-2})(y + y')^2 \big)}\\
\leq & ~ C N^2 e^{-\big( \frac14N^2(x - x')^2 + \frac1{4}N^2(y - y')^2 \big)}.
\end{aligned}
\end{align*}
Therefore, we have
\begin{align*}
\begin{aligned}
 \big\|e^{-\frac{A}{N^2}}f \big\|_{L^q(\R^2)} &\leq CN^2\Big\| \int_{\R^2}e^{-\frac14N^2(z-z')^2}f(z')\dd z' \Big\|_{L^q(\R^2)}\\
&\leq  CN^2\big\| e^{-\frac14N^2z^2} \big\|_{L^r(\R^2)}\|f\|_{L^p(\R^2)} \leq CN^{2(\frac1p-\frac1q)}\|f\|_{L^p(\R^2)},
\end{aligned}
\end{align*}
where we used Young's inequality in the last step.

To prove $(2)$, we only need to prove the first inequality as the second inequality is a direct consequence of Lemma \ref{Theorem-mikhlin multiplier theorem}. Let $\tilde\psi \in C_c^\infty(\mathbb R)$  satisfy $\tilde\psi \psi=\psi$. Then
using \eqref{Bernstein-1} and Lemma \ref{Theorem-mikhlin multiplier theorem} again,
\begin{align*}
    \|\Delta_N u\|_{L^p(\R^2)}&=\Big\|A^{-s/2}\tilde\psi\Big(\frac{A}{N^2}\Big ) A^{s/2}\Delta_N u\Big \|_{L^p(\R^2)}\\
    & \lesssim N^{-s} \|A^{s/2}\Delta_N u\|_{L^p(\R^2)}\\
    &\lesssim N^{-(s+(\frac2p-\frac2q))} \|A^{s/2}\Delta_Nu\|_{L^q(\R^2)}.
\end{align*}
This finishes the proof.
\end{proof}

\begin{remark} The restriction on $p$ and $q$ ($p, q \neq 1, \infty $) in Lemma \ref{Lemma-Bernstein estimate} comes from the restriction of the exponents in Lemma \ref{Theorem-mikhlin multiplier theorem}. We believe the inequalities are valid for $1\leq p\leq q\leq \infty$, but the lemma suffices for our needs.
\end{remark}

Next, for any integer $\lambda \in \mathbb N$,  we denote the functional of  the indicator function of $[\lambda, \lambda+1)$ associated to $A$ as
$\mathbf{1}_{\lambda}$, that is
\begin{align*}
    \mathbf{1}_{\lambda} f(x,y) = \sum_{k=0}^\infty\int_\R e^{ix\xi}\textbf{1}_{ [\lambda, \lambda+1)}(\sqrt{\xi^2+2k+1})\mathcal{F}_{x\to \xi}P_ku(\xi,y)\dd \xi.
\end{align*}
Our next lemma shows that functions spectrally localized in different regions are almost orthogonal.

\begin{lemma}\label{lemma of almost orthogonality}

 Let $m\in \mathbb N$. If $C_0$ is sufficiently large, then for any $\nu>0$,  there exists $C_\nu >0$
  such that for $\lambda_{j}$($j=1,2,3$) satisfies $C_0\lambda_{j}\leq  \lambda_{0}$, then
  \begin{align}\label{product estimate}
      \left|\int_{\R^2} \mathbf{1}_{\lambda_0} f_0 \mathbf{1}_{\lambda_1}f_1 \mathbf{1}_{\lambda_1}f_2  \mathbf{1}_{\lambda_3}f_3 \dd z\right|\leq C_{\nu} \lambda_0^{-\nu } \prod_{j=0}^3\|f_j\|_{L^2(\R^2)}.
  \end{align}
\end{lemma}

\begin{proof}
By definition, the spectrum of $\mathbf{1}_{\lambda_0} f_0 $ is localized near $\lambda_0^2$. To be precise,  $\xi_0^2+2k_0+1\in [\lambda_0^2, (\lambda_0+1)^2)$.  Then either $\xi_0^2 \in [\lambda_0^2/{2}, 2 \lambda_0^2] $ or $2k_0 +1  \in [\lambda_0^2/{2}, 2 \lambda_0^2]$.

 In the first case, since $\xi_j \leq \lambda_j $ for $j=1, 2, 3$, we can choose $C_0$ large enough such that $\sum_{j=1}^3 \lambda_j  < \frac{\lambda_0}{2}<|\xi_0|$. Thus, the Fourier support in $x$ direction of  $\widehat{\mathbf{1}_{\lambda_1} f_1} *\widehat{\mathbf{1}_{\lambda_2} f_2}*\widehat{\mathbf{1}_{\lambda_3} f_3}(\xi, y) \subset \{\xi: |\xi | <|\xi_0|\} $.  Hence, for any $y \in \mathbb R$,
\begin{align*}
    \int_{\mathbb R}\mathbf{1}_{\lambda_0} f_0 \mathbf{1}_{\lambda_1}f_1 \mathbf{1}_{\lambda_1}f_2  \mathbf{1}_{\lambda_3}f_3  d x =\int_{\R}  \widehat{\mathbf{1}_{\lambda_0} f_0}(\xi, y) \widehat{\mathbf{1}_{\lambda_1} f_1} *\widehat{\mathbf{1}_{\lambda_2} f_2}*\widehat{\mathbf{1}_{\lambda_3 }f_3}(\xi, y) d\xi =0.
\end{align*}
Therefore,
\begin{align*}
     \int_{\R^2} \mathbf{1}_{\lambda_0} f_0 \mathbf{1}_{\lambda_1}f_1 \mathbf{1}_{\lambda_1}f_2  \mathbf{1}_{\lambda_3}f_3  \dd z=0.
\end{align*}

 In the second case, if $2k_0  +1 \in [\lambda_0^2/{2}, 2 \lambda_0^2]$, then we have $\sum_{k=1}^3 2k_i+3 \leq  \frac{\lambda_0^2}{2}<2k_0  +1 $. Then by Lemma \ref{Lemma-BTT-lemma 7.5}, for any $y\in \mathbb R$, we have
 \begin{align*}
     \left|\int_{\mathbb R}\mathbf{1}_{\lambda_0} f_0 \mathbf{1}_{\lambda_1}f_1 \mathbf{1}_{\lambda_1}f_2  \mathbf{1}_{\lambda_3}f_3  \dd y\right|\leq C_\nu  \lambda_0^{-\nu} \prod_{i=0}^3\|\mathbf{1}_{\lambda_i} f_i \|_{L_y^2(\mathbb R)}
 \end{align*}
 Hence,
 \begin{align*}
      \left|\int_{\mathbb R^2} \mathbf{1}_{\lambda_0} f_0 \mathbf{1}_{\lambda_1}f_1 \mathbf{1}_{\lambda_1}f_2  \mathbf{1}_{\lambda_3}f_3  \dd z \right|& \leq C_\nu  \lambda_0^{-\nu }\int_{\R}  \prod_{i=0}^3\|\mathbf{1}_{\lambda_i} f_i \|_{L^2_y(\mathbb R)} \dd x  \\
     & \leq C_\nu  \lambda_0^{-\nu}  \prod_{i=0}^3\|\mathbf{1}_{\lambda_i} f_i \|_{L^4_xL_y^2(\R\times\R)}\\
     &\leq C_\nu \lambda_0^{-\nu +2}  \prod_{i=0}^3\|\mathbf{1}_{\lambda_i} f_i \|_{L^2_xL_y^2(\R\times\R)},
 \end{align*}
 where the last inequality follows from the Bernstein inequality.
 Thus, we have finished the proof.
\end{proof}


%

\subsection{The Bourgain space $X^{s,b}$ adapted to $A$}\label{subsection of Bourgain space}
Let $e^{itA}$ be the solution operator for the following linear Schr\"odinger equation
\begin{equation}
\begin{cases}\label{linear equation}
i \partial_t u(t,x,y) + Au(t,x,y) = 0,\quad (t,x,y)\in\R^3,\\
u(0,x,y) = \varphi(x,y).
\end{cases}
\end{equation}
\begin{definition} Let $s\in \mathbb R$. The Bourgain space $X^{s,b}(\mathbb R \times \mathbb R^2)$ is the completion of $C_0^\infty(\mathbb R_t; \H^s(\R^2))$ in the norm
 \begin{align*}
\|u\|_{X^{s,b}(\R\times\R^2)}^2 = & ~ \sum_{k=0}^\infty \big\| \langle\xi^2+2k+1\rangle^{s/2} \langle\tau+\xi^2+(2k+1)\rangle^b\mathcal{F}_{t,x}P_ku(\tau,\xi,y) \big\|_{L^2(\mathbb R_\tau; L^2(\mathbb R^2 ))}^2\\
= & ~ \|e^{itA}u(t,\cdot)\|_{H_t^b(\R; \H^s(\R^2))}^2,
\end{align*}
where $\mathcal{F}_{t,x}$ denotes the Fourier transform with respect to $t$ and  $x$.
\end{definition}
For $0 < T \leq 1$, we denote by $X^{s,b}_T(\R^2)$ the subspace of functions $u\in X^{s,b}(\R\times\R^2)$ such that
\begin{equation}
\|u\|_{X^{s,b}_T(\R^2)} = \inf\big\{ \|\tilde{u}\|_{X^{s,b}(\R\times\R^2)},~ \tilde{u}|_{(-T,T)\times\R^2} = u \big\} < \infty.
\end{equation}
For the Bourgain spaces, we have the following basic properties.
\begin{proposition}\label{Prop-2.2}
\begin{enumerate}
  \item For $s_1 \leq s_2, b_1 \leq b_2$, $X^{s_2,b_2}(\R\times\R^2) \hookrightarrow X^{s_1,b_1}(\R\times\R^2)$.
\item For $b>\frac12$, $X^{s,b}(\R\times\R^2) \hookrightarrow C(\R; \H^s(\R^2))$, and $X^{s,b}_T(\R^2) \hookrightarrow C((-T,T); \H^s(\R^2))$.
\item  For any $\delta\in \mathbb R$, we have
\begin{align*}
  N^{\delta} \|\Delta_N u(t,x,y)\|_{X^{s,b}(\R\times\R^2)} \sim \|\Delta_N u(t,x,y)\|_{X^{s+\delta,b}(\R\times\R^2)}.
\end{align*}
  \item For every $b>\frac{1}{4}$, there exists $C>0$ such that $u\in X^{0,b}(\R\times\R^2)$ satisfying
\begin{align*}
\|u\|_{L^{4}(\R;L^2(\R^2))} \leq C\|u\|_{X^{0,b}(\R\times\R^2)}.
\end{align*}
\end{enumerate}
\end{proposition}
\begin{proof} The proofs are simple and standard, hence we omit the details.
\end{proof}

%
%
\subsection[Local Strichartz estimates]{Local Strichartz estimates (See \cite{AC20})}
For the solution to the linear equation $e^{itA}\varphi$,
the following local Strichartz estimate is valid:
\begin{equation}\label{eq-local Strichartz estimate}
\big\| e^{itA}\varphi \big\|_{L^q(I,L^r(\R^2))} \leq C(I)\|\varphi\|_{L^2(\R^2)},\quad \frac2q = 2\Big( \frac12-\frac1r \Big), \quad 2<q\leq\infty, \quad 2\leq r<\infty,
\end{equation}
where $C(I)$ is a constant depending on $I$.

\section{Local well-posedness in Bourgain spaces}
\label{section3}
In this section, we give the proof of the bilinear estimates \eqref{prop-1.1} and establish a trilinear estimate. Using these two ingredients, we give the local well-posedness for the equation \eqref{equation}.
\subsection{Bilinear estimates} In this subsection, we are going to prove Proposition
\ref{proposition-Poiret2.3.13}. As a first step, we give the following bilinear estimate which is slightly stronger than we need.
\begin{lemma}\label{le-improved bilinear estimates} Let $1 \leq M \leq N$ be dyadic integers, for $T\in(0,\infty)$ there exists $C_T$ such that
   \begin{equation}\label{eq-A^1 estimate}
\int_0^T\|e^{itA}\Delta_N\varphi_1 e^{itA}\Delta_M\varphi_2\|_{\H^1(\mathbb R^2)}^2\dd t \leq C_TMN\|\Delta_N\varphi_1\|_{L^2(\R^2)}^2\|\Delta_M\varphi_2\|_{L^2(\R^2)}^2.
\end{equation}
\end{lemma}
\begin{proof}
Denote $u_N : = e^{itA}\Delta_N\varphi_1, v_M : = e^{-itA}\Delta_M\varphi_2$ for simplicity. Notice that $\bar v_M : = e^{itA}\Delta_M\varphi_2$.

By the equivalence of Sobolev norms  \eqref{eq-equivalence}, to prove \eqref{eq-A^1 estimate},  it suffices to prove the following three inequalities:
\begin{equation}\label{eq-H^1 estimate2}
\int_0^T \int_{\R^2}\big| u_N(z)\nabla_z\bar{v}_M(z) \big|^2\dd z \dd t \leq C_TMN\|\Delta_N\varphi_1\|_{L^2(\R^2)}^2\|\Delta_M\varphi_2\|_{L^2(\R^2)}^2,
\end{equation}
\begin{equation}\label{eq-H^1 estimate1}
\int_0^T \int_{\R^2}\big| \bar{v}_M(z)\nabla_zu_N(z) \big|^2\dd z\dd t \leq C_TMN\|\Delta_N\varphi_1\|_{L^2(\R^2)}^2\|\Delta_M\varphi_2\|_{L^2(\R^2)}^2,
\end{equation}
\begin{equation}\label{eq-H^1 estimate3}
\int_0^T \int_{\R^2}\big| y u_N(z)\bar{v}_M(z) \big|^2\dd z \dd t \leq C_TMN\|\Delta_N\varphi_1\|_{L^2(\R^2)}^2\|\Delta_M\varphi_2\|_{L^2(\R^2)}^2.
\end{equation}

For \eqref{eq-H^1 estimate2},  Cauchy-Schwarz inequality implies
\begin{equation}\label{eq-H^1 estimate21}
\int_0^T \int_{\R^2}\big| u_N(z)\nabla_z\bar{v}_M(z) \big|^2\dd z \dd t \leq \|u_N\|_{L^4((0,T);L^4(\R^2))}^2\|\nabla v_M\|_{L^4((0,T);L^4(\mathbb R^2))}^2.
\end{equation}
Note that $\nabla v_M$ is a solution to the inhomogeneous equation associated with \eqref{linear equation} with nonlinear term $-2yv_M$. Hence by the local Strichartz estimate \eqref{eq-local Strichartz estimate}, we have
\begin{equation}\label{eq-nabla vm}
\begin{aligned}
\|\nabla v_M\|_{L^4((0,T);L^4(\mathbb R^2))} \leq & ~ C_T\|\nabla \Delta_M\varphi_2\|_{L^2(\mathbb R^2)} + C_T\|yv_M\|_{L^1((0,T);L^2(\mathbb R^2))}\\
\leq & ~ C_T\|\Delta_M\varphi_2\|_{\H^1(\mathbb R^2)} \leq C_TM\|\Delta_M\varphi_2\|_{L^2(\mathbb R^2)},
\end{aligned}
\end{equation}
where we used the conservation of $\H^1$ norm of the linear evolution.  Therefore, we have
\begin{align*}
\text{RHS of}\ \eqref{eq-H^1 estimate21} \leq C_TM^2\|\Delta_N\varphi_1\|_{L^2(\R^2)}^2\|\Delta_M\varphi_2\|_{L^2(\R^2)}^2 \leq C_TMN\|\Delta_N\varphi_1\|_{L^2(\R^2)}^2\|\Delta_M\varphi_2\|_{L^2(\R^2)}^2.
\end{align*}
That is, we have proved \eqref{eq-H^1 estimate2}.

Next, we prove \eqref{eq-H^1 estimate3}.
Applying H\"older's inequality to the left  hand side of \eqref{eq-H^1 estimate3}, we have
 \begin{align}\label{eq-h1 estimate3}
\int_0^T\Big( \int_{\R^2}|y|^2|v_Mu_N|^2\dd z \Big)\dd t \leq C\|yv_M\|_{L^4((0,T);L^4(\mathbb R^2))}^2\|u_N\|_{L^4((0,T);L^4(\mathbb R^2))}^2.
\end{align}
Note that  $yv_M$ is a solution to \eqref{linear equation} with nonlinear term $-\nabla v_M $. Then \eqref{eq-local Strichartz estimate} and the conservation of $\H^1$ norm of the linear evolution give
\begin{equation}\label{eq-y^2v_m eqtimate}
\begin{aligned}
\|yv_M\|_{L^4((0,T);L^4(\mathbb R^2))} \leq & ~ C\|y\Delta_M\varphi_2\|_{L^2(\mathbb R^2)} + C\|\nabla v_M\|_{L^1((0,T);L^2(\mathbb R^2))}\\
\leq & ~ \|\Delta_M\varphi_2\|_{\H^1(\mathbb R^2)} \leq C_TM\|\Delta_M\varphi_2\|_{L^2(\mathbb R^2)}.
\end{aligned}
\end{equation}
 Hence, combining  \eqref{eq-h1 estimate3} and \eqref{eq-y^2v_m eqtimate} yields \eqref{eq-H^1 estimate3}.

We now only need to show \eqref{eq-H^1 estimate1}. Using the following inequality (the proof is similar to the proof of Lemma 4.1 in \cite{PTV21}, so we omit the details):
\begin{align*}
|\rho(z)|^2 \leq C\int_{|z-z'|<\lambda^{-1}}(\lambda^{-2}\big| A\rho \big|^2 + \lambda^2 \big| \rho \big|^2)\dd z', ~\forall \rho \in C^\infty(\R^2),
\end{align*}
 we have
\begin{align*}
\begin{aligned}
\MoveEqLeft \int_0^T\Big( \int_{\R^2}\big| \bar{v}_M(z)\nabla_zu_N(z) \big|^2\dd z \Big)\dd t \\ &\leq   C\int_0^T\Big( \iint_{|z-z'|<\frac1M}M^2\big| \bar{v}_M(z')\nabla_z u_N(z) \big|^2
 + \frac{1}{M^2}\big| A\bar{v}_M(z')\nabla_z u_N(z) \big|^2\dd z\dd z' \Big)\dd t.
\end{aligned}
\end{align*}

Adapting the method of \cite{PTV21}, the following Morawetz estimate is valid for the solution to \eqref{linear equation}
 \begin{align}\label{eq-1morewarz reduce estimate}\begin{multlined}
\int_0^T\Big( \iint_{|z-z'|<\frac1M} \big| u_N(z)\nabla_z\bar{v}_M(z') + \nabla_zu_N(z)\bar{v}_M(z') \big|^2\dd z\dd z' \Big)\dd t\\
 \leq C_TM^{-1}N\|\Delta_N\varphi_1\|_{L^2(\R^2)}^2\|\Delta_M\varphi_2\|_{L^2(\R^2)}^2.
 \end{multlined}
\end{align}
In addition, note that $Av_M$ is a localized solution to \eqref{linear equation}.
By \eqref{eq-1morewarz reduce estimate}, we obtain
\begin{equation*}
\begin{multlined}
 \int_0^T \iint_{|z-z'|<\frac1M}\big| u_N(z)\nabla_y(A\bar{v}_M)(y) + \nabla_zu_N(z)A\bar{v}_M(z') \big|^2\dd z\dd z' \dd t \\
\leq  C_TNM\|\Delta_N\varphi_1\|_{L^2(\R^2)}^2\|\Delta_M\varphi_2\|_{L^2(\R^2)}^2.
\end{multlined}
\end{equation*}
Therefore,
\begin{align*}
\begin{aligned}
 \MoveEqLeft \int_0^T\Big( \int_{\R^2}\big| \bar{v}_M(z)\nabla_zu_N(z) \big|^2\dd z \Big)\dd t
\\ &\leq  C\int_0^T\Big( \iint_{|z-z'|<\frac1M}M^2 \big| u_N(z)\nabla_y\bar{v}_M(z')\big|^2 + \frac1{M^2}\big| u_N(z)\nabla_{z'}(A\bar{v}_M)(z')\big|^2\dd z\dd z' \Big)\dd t\\
& \quad + C_TNM\|\Delta_N\varphi_1\|_{L^2(\R^2)}^2\|\Delta_M\varphi_2\|_{L^2(\R^2)}^2.
\end{aligned}
\end{align*}
Using the Cauchy-Schwarz inequality, local Strichartz estimates \eqref{eq-local Strichartz estimate} and \eqref{eq-nabla vm}, we have
\begin{equation}\label{eq-v_M estimate}
\begin{aligned}
\MoveEqLeft \int_0^T\Big( \iint_{|z-z'|<\frac1M} M^2 \big| u_N(z)\nabla_y\bar{v}_M(z')\big|^2\dd z\dd z' \Big)\dd t\\
 \leq & ~M^2 \int_0^T\int_{|w|<\frac1M}\big| u_N(z)\nabla_z\bar{v}_M(w-z)\big|^2\dd z\dd w\dd t\\
\leq & ~M^2\int_{|w|<\frac1M}\|u_N\|_{L^4((0,T);L^4(\R^2))}^2\|\nabla v_M\|_{L^4((0,T);L^4(\R^2))}^2\dd w\\
\leq & ~ C_T M^2\|\Delta_N\varphi_1\|_{L^2(\R^2)}^2\|\Delta_M\varphi_2\|_{L^2(\R^2)}^2.
\end{aligned}
\end{equation}
Replacing $v_M$ by $Av_M$ in \eqref{eq-v_M estimate}, we get
\begin{align}\label{eq-Av_M estimate}
\int_0^T\!\!\Big( \iint\limits_{|z-z'|<\frac1M}\frac{1}{M^2} \big| u_N(z)\nabla_y(A\bar{v}_M)(z')\big|^2\dd z\dd z' \Big)\dd t
& \leq  C_TM^2\|\Delta_N\varphi_1\|_{L^2(\R^2)}^2\|\Delta_M\varphi_2\|_{L^2(\R^2)}^2.
\end{align}
Combining \eqref{eq-v_M estimate} and \eqref{eq-Av_M estimate} yields \eqref{eq-H^1 estimate1}. This finishes the proof of the lemma.
\end{proof}

\begin{proposition}
\label{Theorem-linear bilinear estimate}
Let $1 \leq M \leq N$ be dyadic integers. For $T\in(0,\infty)$ there exists $C_T$ such that
\begin{align*}
\big\| e^{itA}\Delta_N\varphi_1 e^{itA}\Delta_M\varphi_2 \big\|_{L^2((0,T);L^2(\R^2))}^2 \leq C_TMN^{-1}\|\Delta_N\varphi_1\|_{L^2(\R^2)}^2\|\Delta_M\varphi_2\|_{L^2(\R^2)}^2.
\end{align*}
\end{proposition}
\begin{proof}
Denote $u_N : = e^{itA}\Delta_N\varphi_1$, $v_M : = e^{itA}\Delta_M\varphi_2$ for simplicity.
We perform a Littlewood--Paley decomposition for $v_M$ and $ u_N$, and divide the sum in $K$ into $K \leq N$ and $K \geq N$. Then
\begin{align*}
&\quad\|v_Mu_N\|_{L^2((0,T);L^2(\mathbb R^2))}^2 = \Big\|\sum_{K\in2^\N}\Delta_K(v_Mu_N)\Big\|_{L^2((0,T);L^2(\mathbb R^2))}^2\\
&\leq \Big\|\sum_{K\in2^\N, K\leq N}\Delta_K(v_Mu_N)\Big\|_{L^2((0,T);L^2(\mathbb R^2))}^2+\Big\|\sum_{K\in2^\N, K\geq  N}\Delta_K(v_Mu_N)\Big\|_{L^2((0,T);L^2(\mathbb R^2))}^2\\
&\eqqcolon I_1+I_2.
\end{align*}
\textbf{Low  frequency:} For $I_1$, using Lemma \ref{Theorem-mikhlin multiplier theorem}, we have
\begin{align*}
\quad I_1 &\leq  \| S_N(v_Mu_N)\|_{L^2(\mathbb R^2)}^2  \leq N^{-2} \| N A^{-1/2} S_N A^{1/2}(v_Mu_N) \|_{L^2(\mathbb R^2)}^2 \leq  N^{-2} \| A^{1/2}(v_Mu_N) \|_{L^2(\mathbb R^2)}^2 \\
& \leq N^{-2} \| \nabla (v_Mu_N) \|_{L^2(\mathbb R^2)}^2+ N^{-2}\| y (v_Mu_N) \|_{L^2(\mathbb R^2)}^2.
\end{align*}

Therefore, by \eqref{eq-A^1 estimate}, \eqref{eq-nabla vm} and \eqref{eq-y^2v_m eqtimate}, we obtain
\begin{align*}
\begin{aligned}
\int_0^T\big\| S_N(v_Mu_N) \big\|_{L^2(\mathbb R^2)}^2\dd t
\leq & ~C(MN^{-1} + M^2N^{-2})\|\Delta_N \varphi_1\|_{L^2(\R^2)}^2\|\Delta_M\varphi_2\|_{L^2(\R^2)}^2.
\end{aligned}
\end{align*}
\textbf{High  frequency:} For $I_2$, by the triangle inequality and  Lemma \ref{eq-A^1 estimate}, we have
\begin{align*}
\begin{aligned}
I_2 \leq \sum_{K\in2^\N}\big\|\Delta_K(v_Mu_N)\big\|_{L^2((0,T);L^2(\mathbb R^2))}^2 \leq & ~ C\sum_{K\in2^\N}(1+K)^{-2}\|\Delta_K (v_Mu_N)\|_{L^2((0,T);\H^1(\mathbb R^2))}^2\\
\leq & ~ C_TMN^{-1}\|\Delta_N\varphi_1\|_{L^2(\R^2)}^2\|\Delta_M\varphi_2\|_{L^2(\R^2)}^2.
\end{aligned}
\end{align*}
 This completes the proof.
\end{proof}

\begin{proof}[\bf{Proof of Proposition \ref{proposition-Poiret2.3.13}}]
\label{Proposition-Poiret2.3.15}

First, we claim that Lemma \ref{le-improved bilinear estimates} implies the following bilinear estimate in Bourgain space: for every $\frac12<b\leq1$ and $\delta>0$, and let $1 \leq M \leq N$ be dyadic integers. Then there exists $C>0$ such that for every $u,v\in X^{0,b}(\R\times\R^2)$ satisfying
\begin{equation}\label{eq-local bilinear estimate}
\big\| \Delta_{N}u\Delta_{M}v \big\|_{L^2([0,1];L^2(\R^2))} \leq C\big(\frac{M}{N}\big)^{\frac12-\delta}\big\| \Delta_{N}u \big\|_{X^{0,b}(\R\times\R^2)}\big\| \Delta_{M}v \big\|_{X^{0,b}(\R\times\R^2)}.
\end{equation}
To prove this, let us set $U(t) = e^{itA}\Delta_{N}u(t)$, $V(t) = e^{itA}\Delta_{M}v(t)$, then we write
\begin{align*}
\Delta_{N}u(t) = \frac{1}{2\pi}\int_\R e^{it\tau}e^{-itA}\hat{U}(\tau)\dd \tau,\quad
\Delta_{M}v(t) = \frac{1}{2\pi}\int_\R e^{it\tau}e^{-itA}\hat{V}(\tau)\dd \tau.
\end{align*}
Therefore,
\begin{align*}
\Delta_{N}u(t)\Delta_{M}v(t) = \frac{1}{4\pi^2}\int_\R\int_\R e^{it(\tau_1+\tau_2)}\big(e^{-itA}\hat{U}(\tau_1)e^{-itA}\hat{V}(\tau_2)\big)\dd \tau_1\dd \tau_2.
\end{align*}
Using the triangle inequality and Proposition \ref{Theorem-linear bilinear estimate}, we have that for every unit interval and $0<\delta\leq\frac12$, there exists $C >0$ such that for $b>1/2$
\begin{align*}
\begin{aligned}
\MoveEqLeft \big\| \Delta_{N}u\Delta_{M}v \big\|_{L^2([0,1];L^2(\mathbb R^2))} \leq C\Big\| \int_\R\int_\R e^{it(\tau_1+\tau_2)}\big(e^{-itA}\hat{U}(\tau_1)e^{-itA}\hat{V}(\tau_2)\big)\dd \tau_1\dd \tau_2 \Big\|_{L^2([0,1];L^2(\mathbb R^2))}\\
\leq & ~ C\int_\R\int_\R \Big\| e^{-itA}\hat{U}(\tau_1)e^{-itA}\hat{V}(\tau_2) \Big\|_{L^2([0,1];L^2(\mathbb R^2))}\dd \tau_1\dd \tau_2\\
\leq & ~ C\big(\frac{M}{N}\big)^{\frac12-\delta}\int_\R\int_\R \big\| \hat{U}(\tau_1) \big\|_{L^2(\R^2)}\big\| \hat{V}(\tau_2) \big\|_{L^2(\R^2)}\dd \tau_1\dd \tau_2\\
\leq  & ~ C \big(\frac{M}{N}\big)^{\frac12-\delta}\big\| \langle\tau \rangle^b\hat{U}(\tau_1) \big\|_{L^2(\R^2)}\big\| \langle\tau \rangle^b\hat{V}(\tau_2) \big\|_{L^2(\R^2)}\dd \tau_1\dd \tau_2\\
\leq & ~ C\big(\frac{M}{N}\big)^{\frac12-\delta}\big\| \Delta_{N}u \big\|_{X^{0,b}(\R\times\R^2)}\big\| \Delta_{M}v \big\|_{X^{0,b}(\R\times\R^2)}.
\end{aligned}
\end{align*} Then using a partition of unity, we have the following global in time estimate:
\begin{align}\label{prop1.1-interpolation 1}
\big\| \Delta_{N}u\Delta_{M}v \big\|_{L^2(\R;L^2(\R^2))} \leq C \big(\frac{M}{N}\big)^{\frac12-\delta}\big\| \Delta_{N}u \big\|_{X^{0,b}(\R\times\R^2)}\big\| \Delta_{M}v \big\|_{X^{0,b}(\R\times\R^2)}.
\end{align}

On the other hand,
for any $\varepsilon>0$, by H\"older's inequality and  Proposition \ref{Prop-2.2}, we have
\begin{align}\label{prop1.1-interpolation 2}
\begin{aligned}
& \big\| \Delta_{N}u\Delta_{M}v \big\|_{L^2(\R;L^2(\R^2))} \leq C \|\Delta_{N}u\|_{L^4(\R;L^2(\R^2))}\|\Delta_{M}v\|_{L^4(\R;L^\infty(\R^2))}\\
\leq & ~ C \|\Delta_{N}u\|_{X^{0,\frac14+\varepsilon}}\|\Delta_{M}v\|_{X^{1+\varepsilon,\frac14+\varepsilon}} \leq C M^{1+\varepsilon}\|\Delta_{N}u\|_{X^{0,\frac14+\varepsilon}}\|\Delta_{M}v\|_{X^{0,\frac14+\varepsilon}}.
\end{aligned}
\end{align}
Now, interpolating \eqref{prop1.1-interpolation 1} and \eqref{prop1.1-interpolation 2}, we have for any $\eta\in[0,1]$,
\begin{align*}
\begin{aligned}
& \big\| \Delta_{N}u\Delta_{M}v \big\|_{L^2(\R;L^2(\R^2))}\\
\leq & ~ CM^{\eta(1+\varepsilon)}\Big( \frac{M}{N} \Big)^{(\frac12-\delta)(1-\eta)}\big\| \Delta_{N}u \big\|_{X^{0,b(1-\eta)+\eta(\frac14+\varepsilon)}}\big\| \Delta_{M}v \big\|_{X^{0,b(1-\eta)+\eta(\frac14+\varepsilon)}}.
\end{aligned}
\end{align*}
Choosing $\delta=\frac{\varepsilon}{2}$ and $\eta=\frac{\varepsilon}{4}$, then
\begin{align*}
b(1-\eta)+\eta\Big(\frac14+\varepsilon\Big) = b-\frac{b\varepsilon}{4}+\frac{\varepsilon}{4}\Big(\frac14+\varepsilon\Big)\leq b-\frac{\varepsilon}{8}+\frac{\varepsilon}{16}+ \frac{\varepsilon^2}{4}\leq b-\frac{\varepsilon}{17}<\frac12,
\end{align*}
we complete the proof of the proposition.
\end{proof}

\subsection{Trilinear estimate}
In this subsection, we are going to prove a trilinear estimate in Bourgain spaces using the bilinear estimates and the almost-orthogonality for the products of functions  spectrally localized to different regions.

\begin{proposition}\label{prop-trilinear estimate}
Let $\varepsilon > 0$. There exists $C>0, b$ and $b'$ satisfying  $0<b'<1/2<b, \ b+b'<1,$ such that for any $s \geq \varepsilon$ and for every triple $(u_j), j=1, 2, 3$  in $X^{s,b}$, we have
\begin{equation}\label{trilinear estimate}
\big\|u_1  u_2\overline u_3 \big\|_{X^{s,-b'}_T(\R^2)} \leq C \|u_1\|_{X^{s,b}_T(\R^2)}\|u_2\|_{X^{\varepsilon,b}_T(\R^2)}\|u_3\|_{X^{\varepsilon,b}_T(\R^2)}.
\end{equation}
\end{proposition}
\begin{proof}
    By duality, to prove \eqref{trilinear estimate} is equivalent to prove
\begin{equation}
\Big| \int_0^T\int_{\R^2}u_1u_2\bar{u}_3\bar{u}_0\dd z\dd t \Big| \leq C \|u_0\|_{X^{-s,b'}_T(\R^2)} \|u_1\|_{X^{s,b}_T(\R^2)}\|u_2\|_{X^{\varepsilon,b}_T(\R^2)}\|u_3\|_{X^{\varepsilon,b}_T(\R^2)}.
\end{equation}
 Then performing a Littlewood--Paley decomposition on each function $u_i$ and using a symmetry argument, it suffices to estimate
\begin{align*}
\Big| \sum_{N_1\geq N_2\geq N_3}\sum_{N_0}\int_0^T\int_{\R^2}\Delta_{N_0}\bar{u}_0\Delta_{N_1}u_1\Delta_{N_2}u_2\Delta_{N_3}\bar{u}_3\dd z\dd t \Big|,
\end{align*}
where the summation is meant over dyadic values of $N_1, N_2, N_3$ and $N_0$. The other possible orders of magnitude of $N_1, N_2, N_3$ are similar.

\textbf{Case 1:} $N_0\geq N_1^{1+\delta}$ for some $\delta>0$.
Since \(\Delta_N=\sum_{k \in [\frac{N^2}{2}, 2N^2]} \textbf{1}_{k}\Delta_N,\)
\begin{align*}
 \MoveEqLeft \Delta_{N_0}\bar{u}_0\Delta_{N_1}u_1\Delta_{N_2}u_2\Delta_{N_3}\bar{u}_3
 =  \sum_{i=0}^3\sum_{k_i \in [\frac{N_i^2}{2}, 2N_i^2 ]} \textbf{1}_{k_0}\Delta_{N_0}\bar{u}_0 \textbf{1}_{k_1} \Delta_{N_1}u_1 \textbf{1}_{k_2} \Delta_{N_1} u_2 \textbf{1}_{k_3}  \Delta_{N_1}\bar{u}_{k_3}.
\end{align*}
Thus by Lemma \ref{lemma of almost orthogonality}, we have
\begin{align*}
   & \left |\int_{\R^2}\Delta_{N_0}\bar{u}_0\Delta_{N_1}u_1\Delta_{N_2}u_2\Delta_{N_3}\bar{u}_3 d z  \right |\\
   &\leq \prod_{i=0}^3 N_i^2 N_0^{-\nu }\|\Delta_{N_0}\bar{u}_0(t, \cdot) \|_{L^2(\mathbb R^2)}\|\Delta_{N_1}{u}_1(t, \cdot) \|_{L^2(\mathbb R^2)}\|\Delta_{N_2}{u}_2(t, \cdot) \|_{L^2(\mathbb R^2)}\|\Delta_{N_3}\bar{u}_3(t, \cdot) \|_{L^2(\mathbb R^2)}.
\end{align*}
Integrating in the variable $t$, using H\"older's inequality and $X^{0, b'} \hookrightarrow L^4_t L^2_z$, we have
\begin{align*}
    \Big|\int_0^T\int_{\R^2}\Delta_{N_0}\bar{u}_0\Delta_{N_1}u_1\Delta_{N_2}u_2\Delta_{N_3}\bar{u}_3\dd z\dd t \Big| \leq C_K N_0^{-K}\prod_{i=0}^3\big\| \Delta_{N_i}u_i \big\|_{X^{0,b'}_T(\R^2)},
\end{align*}

where $b'<\frac12$. The sum in $N_0,N_1,N_2,N_3$ is finite due to the large negative power of $N_0$.

\textbf{Case 2:} $N_0\leq N_1^{1+\delta}$, with $\delta>0$ to be chosen later depending on $\varepsilon$. By Proposition \ref{proposition-Poiret2.3.13},
\begin{align*}
\begin{aligned}
& \Big| \int_0^T\int_{\R^2}\Delta_{N_0}\bar{u}_0\Delta_{N_1}u_1\Delta_{N_2}u_2\Delta_{N_3}\bar{u}_3\dd z\dd t \Big|\\
\leq & \big\| \Delta_{N_1}u_1\Delta_{N_2}u_2 \big\|_{L^2((0,T);L^2(\R^2))}\big\| \Delta_{N_0}u_0\Delta_{N_3}u_3 \big\|_{L^2((0,T);L^2(\R^2))}\\
\leq & ~ C(N_2N_3)^\delta\Big( \frac{N_2N_3}{N_0N_1} \Big)^{\frac12-\delta}\prod_{j=0}^3\big\| \Delta_{N_j}u_j \big\|_{X^{0,b'}_T(\R^2)}.
\end{aligned}
\end{align*}
Since $N_0\leq N_1^{1+\delta}$, we have for $\delta>0$ sufficiently small so that $\kappa=2\epsilon - 2\delta$ is positive,
\begin{align*}
\begin{aligned}
(N_2N_3)^\delta\Big( \frac{N_2N_3}{N_0N_1} \Big)^{\frac12-\delta}N_0^sN_1^{-s}(N_2N_3)^{-\varepsilon}
&=  N_0^{s-\frac12+\delta}(N_2N_3)^{\frac12-\varepsilon}N_1^{-s-\frac12+\delta}\\
&\leq  N_1^{-2\varepsilon+2\delta} = N_1^{-\kappa},
\end{aligned}
\end{align*}
so we can sum over $N_0,N_1,N_2,N_3$.
Using  property $(3)$ of Proposition \ref{Prop-2.2}, we have
\begin{align*}
\begin{aligned}
\MoveEqLeft \sum_{\substack{
N_0\leq N_1^{1+\delta} \\
N_1\geq N_2\geq N_3
}}
\!\! (N_2N_3)^\delta\Big( \frac{N_2N_3}{N_0N_1} \Big)^{\frac12-\delta}N_0^sN_1^{-s}(N_2N_3)^{-\varepsilon} \big\| \Delta_{N_0}u_0 \big\|_{X^{-s,b'}_T(\R^2)}\big\| \Delta_{N_1}u_1 \big\|_{X^{s,b'}_T(\R^2)}\prod_{j=2}^3\big\| \Delta_{N_j}u_j \big\|_{X^{\varepsilon,b'}_T(\R^2)}\\
& \leq  ~ C\big( \sum_N\big\| \Delta_{N}u \big\|_{X^{-s,b'}_T(\R^2)} \big)^\frac12\big( \sum_N\big\| \Delta_{N}u \big\|_{X^{s,b'}_T(\R^2)} \big)^\frac12\prod_{j=2}^3\big( \sum_{N_j}\big\| \Delta_{N_j}u_j \big\|_{X^{\varepsilon,b'}_T(\R^2)} \big)^\frac12.
\end{aligned}
\end{align*}
 This completes the proof.
\end{proof}
\subsection{Local well-posedness in Bourgain spaces}
In this section, we give the well-posedness for the equation \eqref{equation} in Bourgain spaces.
%

\begin{theorem}\label{Theorem-well-posedness}
Let $R>0$ and $s_0\geq \varepsilon>0$ be given. Then there exist $T>0$ and $b>\frac12$ such that \eqref{equation} has a unique local solution in $X^{s_0,b}_T(\R^2)$ for every $\varphi\in \H^{s_0}(\R^2)$ with $\|\varphi\|_{\H^1(\R^2)}<R$. Moreover,
\begin{align*}
\|u\|_{X^{s_0,b}_T(\R^2)} \leq C\|\varphi\|_{\H^{s_0}(\R^2)}.
\end{align*}
\end{theorem}
\begin{proof}
To prove this, we apply a fixed point argument for the integral equation
\begin{equation}\label{eq-duhamel}
u(t) = e^{itA}\varphi \pm i\int_0^te^{i(t-\tau)A}(|u|^2u)(\tau)\dd \tau.
\end{equation}
Let $\chi\in C_0^\infty(\R)$ be a cut-off function such that $\operatorname{supp} \chi \subset(-2,2)$ and $\chi\equiv1$ on the interval $[-1,1]$. For $T\leq1$, we consider the following truncated integral equation
\begin{equation}\label{eq-truncated duhamel}
u(t) = \chi(t)e^{itA}\varphi \pm i\chi\left(\frac tT\right)\int_0^te^{i(t-\tau)A}(|u|^2u)(\tau)\dd \tau.
\end{equation}
By the definition of $X^{s_0,b}$, we have
\begin{align*}
\|\chi(t)e^{itA}\varphi\|_{X^{s_0,b}_T(\R^2)} = \|\chi(t)\|_{H^b_t(\R)}\|\varphi\|_{\H^{s_0}(\R^2)} \leq C\|\varphi\|_{\H^{s_0}(\R^2)}.
\end{align*}
It follows that
\begin{align*}
\Big\| \chi\left(\frac tT\right)\int_0^t e^{i(t-\tau)A}\big(|u|^2u\big)(\tau)\dd \tau \Big\|_{X^{s_0,b}_T(\R^2)} \leq C T^{1-b-b'}\big\| |u|^2u \big\|_{X^{s_0,-b'}_T(\R^2)}
\end{align*}
holds for $T\leq1$ and $b>\frac12, b'<\frac12$ such that $b+b'<1$. By Proposition \ref{prop-trilinear estimate}, we obtain
\begin{equation}\label{eq-multilinear estimate}
\big\| |u|^2u \big\|_{X^{s_0,-b'}_T(\R^2)} \leq C \|u\|_{X^{s_0,b}_T(\R^2)}\|u\|_{X^{\varepsilon,b}_T(\R^2)}^2
\end{equation}
for some $b>\frac12, b'<\frac12$ such that $b+b'<1$. Note that $T\leq1$, the nonlinear part can be absorbed by the left-hand side and thus we have completed the proof.
\end{proof}

\section{Growth of Sobolev norms}\label{sec-growth}
In this section, we will  write $\H^s$ instead of  $\H^s(\mathbb R^2)$ for the ease of notation.
This section is devoted to proving Theorem \ref{Theorem-growth}. Since $\partial_t$ has better commutation properties with the nonlinear Schr\"odinger flow than the operator $A$, we first prove that the $\H^{2k}$ norm of the solution $u$ and $\|\partial_tu\|_{L^2}$ are comparable.

\begin{proposition}\label{Proposition-partial_t estimate}
Let $k,s\in\N$ and $R>0$ be given. Let $T>0$ be associated with $R$ and $s_0 = 2k+s$ as in Theorem \ref{Theorem-well-posedness} and let $u\in X^{2k+s,b}_T(\R^2)$ be the unique solution to \eqref{equation} with initial condition $\varphi\in \H^{2k+s}$ with $\|\varphi\|_{\H^1}<R$. Assume moreover that $\sup_{(-T,T)}\|u(t, \cdot)\|_{\H^1}<R$. Then there exists $C>0$ such that
\begin{equation}\label{eq-comparable}
\big\| \partial_t^ku(t)-i^kA^ku(t) \big\|_{\H^s} \leq C\|u(t)\|_{\H^{s+2k-1}}, \quad \forall~t\in(-T,T).
\end{equation}
\end{proposition}
\begin{proof}
Fixed $t$, and using the induction argument on $m$, we obtain
\begin{equation}\label{eq-induction}
\partial_t^mu = i^mA^mu + \sum_{j=0}^{m-1}c_j\partial_t^jA^{m-j-1}(|u|^2u),\quad \forall m\geq1
\end{equation}
for some $c_j\in\C$.

Next, we use the induction on $k$ to prove \eqref{eq-comparable}. In fact, for $k=0$, it is easy to verify. Therefore, assuming \eqref{eq-comparable} holds for $k$, we only need to show that the same estimate is true for $k+1$. Using \eqref{eq-induction} with $m=k+1$, we want to show\begin{align*}
\begin{aligned}
    \big\| \partial_t^{k+1}u(t)-i^{k+1}A^{k+1}u(t) \big\|_{\H^s}
    &= \big\| \sum_{j=0}^{k}c_j\partial_t^jA^{k-j}(|u|^2u) \big\|_{\H^s}\\
    &\leq   \sum_{j=0}^{k}\big\| \partial_t^j(|u|^2u) \big\|_{\H^{2k-2j+s}} \leq C \|u\|_{\H^{s+2k+1}}.
\end{aligned}
\end{align*}
Then by \eqref{eq-equivalence}, it suffices to prove
\begin{equation}\label{eq-derivative}
\big\| D_{x,y}^{2k-2j+s}\partial_t^j(|u|^2u) \big\|_{L^2(\R^2)} \leq C \|u\|_{\H^{s+2k+1}}, \quad j=0,1,...,k,
\end{equation}
and
\begin{equation}\label{eq-moment}
\big\| \langle y \rangle^{2k-2j+s}\partial_t^j(|u|^2u) \big\|_{L^2(\R^2)} \leq C \|u\|_{\H^{s+2k+1}}, \quad j=0,1,...,k.
\end{equation}

For \eqref{eq-derivative}, we use H\"older's inequality, Sobolev embedding, and the induction hypothesis to write
\begin{equation}\label{eq-derivative1}
\begin{aligned}
\sum_{\substack{j_1+j_2+j_3=j,\\ s_1+s_2+s_3=2k-2j+s}}\prod_{l=1,2,3} \big\| \partial^{j_l}_t u \big\|_{W^{s_l,6}(\R^2)} \leq & ~C \sum_{\substack{j_1+j_2+j_3=j,\\ s_1+s_2+s_3=2k-2j+s}}\prod_{l=1,2,3} \big\| \partial^{j_l}_t u \big\|_{\H^{s_l+1}}\\
\leq & ~C\sum_{\substack{j_1+j_2+j_3=j,\\ s_1+s_2+s_3=2k-2j+s}}\prod_{l=1,2,3} \| u \|_{\H^{2j_l+s_l+1}},
\end{aligned}
\end{equation}
where $\|u\|_{W^{s,p}(\R^2)} = \|\langle \nabla \rangle^su\|_{L^p(\R^2)}$. The interpolation argument implies
\begin{align*}
\eqref{eq-derivative1} \leq C \sum_{\substack{j_1+j_2+j_3=j,\\ s_1+s_2+s_3=2k-2j+s}}\big( \prod_{l=1,2,3}\| u \|_{\H^{s+2k+1}}^{\theta_l}\| u \|_{\H^{s+2k+1}}^{1-\theta_l} \big),
\end{align*}
where $\theta_l(s+2k+1)+(1-\theta_l) = 2j_l+s_l+1$ and $\sum_{l=1}^3\theta_l=1$ for every $j=0,...,k$. Therefore, we conclude the proof of \eqref{eq-derivative}.

Now let us prove \eqref{eq-moment}. By Leibniz rule, H\"older's inequality and the Sobolev embedding, we have
\begin{equation}
\begin{aligned}\label{eq-moment1}
\big\| \langle y \rangle^{2k-2j+s}\partial_t^j(|u|^2u) \big\|_{L^2(\R^2)}  \leq & ~ \sum_{\substack{j_1+j_2+j_3=j,\\ s_1+s_2+s_3=2k-2j+s}}\prod_{l=1,2,3} \big\| \langle y \rangle^{s_l}\partial^{j_l}_t u \big\|_{L^{6}(\R^2)}\\
\leq & C \sum_{\substack{j_1+j_2+j_3=j,\\ s_1+s_2+s_3=2k-2j+s}}\prod_{l=1,2,3} \big\| \partial^{j_l}_t u \big\|_{\H^{s_l+1}}.
\end{aligned}
\end{equation}
Then similarly to the proof of \eqref{eq-derivative}, we  get \eqref{eq-moment}.
\end{proof}

We will estimate the norm of $\partial_t$ of the solution $u$ in the localized Bourgain $X^{s,b}_T(\R^2)$ spaces, by using suitable Sobolev norms of the initial data. Before that, let us first give some estimates about suitable linear operators in the $X^{s,b}(\R\times\R^2)$ spaces.

\begin{lemma}\label{Proposition-bougain space estimate}
For every $\delta\in(0,\frac12)$, $b\in(0,1)$, there exists $C>0$ such that we have the following estimates for every $T>0$:
\begin{align}\label{eq-linear estimate}
\|Lu\|_{X^{-\frac12+\delta,\frac12-\delta+2\delta b}_T(\R^2)} \leq & ~ C \|u\|_{X^{\frac12+\delta,\frac12-\delta+2\delta b}_T(\R^2)},\\
\label{eq-linear estimate1}
\|Lu\|_{X^{\delta,(1-\delta)b}_T(\R^2)} \leq & ~ C \|u\|_{X^{1+\delta,(1-\delta)b}_T(\R^2)},
\end{align}
where $L$ can be either $\partial_{x}$, $\partial_{y}$ or multiplication by $\langle y \rangle$.
\end{lemma}
\begin{proof}
The proof is same as the proof of Proposition 2.1 in \cite{PTV21}, hence we omit it.
\end{proof}

\begin{proposition}\label{Proposition-strichartz estimate}
Let $n\in\N$, $ R>0$ and $s\in (0,2]$ be given. Let $T>0$ be associated with $R$ and $s_0 = 2n+2$ as in Theorem \ref{Theorem-well-posedness} and let $u\in X^{2n+2,b}_T(\R^2)$ be the unique solution to \eqref{equation} with initial condition $\varphi\in \H^{2n+2}$ and $\|\varphi\|_{\H^1}<R$. Assume that $\sup_{t\in(-T,T)}\|u(t, \cdot)\|_{\H^1}<R$. Then there exists $C>0$ such that
\begin{equation}\label{eq-Strichartz estimate}
\big\| \partial_t^nu \big\|_{X^{s,b}_T(\R^2)} \leq C \|\varphi\|_{\H^{2n}}^{1-s}\|\varphi\|_{\H^{2n+1}}^s, \quad s\in(0,1],
\end{equation}
and
\begin{equation}\label{eq-Strichartz estimate1}
\big\| \partial_t^nu \big\|_{X^{s,b}_T(\R^2)} \leq C \|\varphi\|_{\H^{2n+1}}^{2-s}\|\varphi\|_{\H^{2n+2}}^{s-1}, \quad s\in(1,2].
\end{equation}
\end{proposition}
\begin{proof}
We first consider the case where $s\in(0,1]$. By the integral equation solved by $\partial_t^nu$,
\begin{align*}
\partial_t^nu(t) = e^{itA}\partial_t^nu(0) + \int_0^te^{i(t-\tau)A}\partial_\tau^n(|u|^2u)(\tau)\dd \tau,
\end{align*}
and the standard properties in Bourgain space $X^{s,b}$, we obtain
\begin{align*}
\big\| \partial_t^nu(t) \big\|_{X^{s,b}_T(\R^2)} \leq C \Big( \big\| \partial_t^nu(0) \big\|_{\H^s} + \Big\| \int_0^te^{i(t-\tau)A}\partial_\tau^n(|u|^2u)(\tau)\dd \tau \Big\|_{X^{s,b}_T} \Big).
\end{align*}
Then by Leibniz rule and Proposition \ref{prop-trilinear estimate}, we obtain
\begin{equation}\label{eq-partial_lu estimate}
\big\| \partial_t^nu(t) \big\|_{X^{s,b}_T(\R^2)}
\leq C \Big( \big\| \partial_t^nu(0) \big\|_{\H^s}
+ T^\gamma\sum_{n_1+n_2+n_3=n}\big\| \partial_t^{n_1}u(t) \big\|_{X^{s,b}_T}\big\| \partial_t^{n_2}u(t) \big\|_{X^{s,b}_T}\big\| \partial_t^{n_3}u(t) \big\|_{X^{s,b}_T} \Big).
\end{equation}

We first estimate the linear part $\big\| \partial_t^n u(0) \big\|_{\H^s}$. By interpolation and Proposition \ref{Proposition-partial_t estimate},
\begin{equation}
\begin{aligned}\label{eq-partiall linear estimate}
\big\| \partial_t^nu(0) \big\|_{\H^s} \leq & ~ C \big\| \partial_t^nu(0) \big\|_{\H^1}^s\big\| \partial_t^nu(0) \big\|_{L^2(\R^2)}^{1-s} \leq C\| \varphi \|_{\H^{2n+1}}^s\| \varphi \|_{\H^{2n}}^{1-s}.
\end{aligned}
\end{equation}
Then we need to estimate the nonlinear part in \eqref{eq-partial_lu estimate}. In fact
\begin{align*}
& \sum_{n_1+n_2+n_3=n}\big\| \partial_t^{n_1}u(t) \big\|_{X^{s,b}_T(\R^2)}\big\| \partial_t^{n_2}u(t) \big\|_{X^{s,b}_T(\R^2)}\big\| \partial_t^{n_3}u(t) \big\|_{X^{s,b}_T(\R^2)} \\
\leq & \sum_{\substack{n_1+n_2+n_3=n,\\ 0< \min\{ n_1,n_2,n_3 \}\leq \max\{ n_1,n_2,n_3 \}<n}}\big\| \partial_t^{n_1}u(t) \big\|_{X^{s,b}_T(\R^2)}\big\| \partial_t^{n_2}u(t) \big\|_{X^{s,b}_T(\R^2)}\big\| \partial_t^{n_3}u(t) \big\|_{X^{s,b}_T(\R^2)}\\
& + \sum_{\substack{n_1+n_2+n_3=n,\\ 0= \min\{ n_1,n_2,n_3 \}\leq \max\{ n_1,n_2,n_3 \}<n}}\big\| \partial_t^{n_1}u(t) \big\|_{X^{s,b}_T(\R^2)}\big\| \partial_t^{n_2}u(t) \big\|_{X^{s,b}_T(\R^2)}\big\| \partial_t^{n_3}u(t) \big\|_{X^{s,b}_T(\R^2)}\\
& + \sum_{\substack{n_1+n_2+n_3=n,\\ \max\{ n_1,n_2,n_3 \}=n}}\big\| \partial_t^{n_1}u(t) \big\|_{X^{s,b}_T(\R^2)}\big\| \partial_t^{n_2}u(t) \big\|_{X^{s,b}_T(\R^2)}\big\| \partial_t^{n_3}u(t) \big\|_{X^{s,b}_T(\R^2)} = I + II + III.
\end{align*}

For $I$, by the induction hypothesis for $n_1,n_2,n_3$ and interpolation, we have
  \begin{align*}
  \begin{aligned}
  & \sum_{\substack{n_1+n_2+n_3=n,\\ \max\{ n_1,n_2,n_3 \}<n}}\big\| \partial_t^{n_1}u(t) \big\|_{X^{s,b}_T(\R^2)}\big\| \partial_t^{n_2}u(t) \big\|_{X^{s,b}_T(\R^2)}\big\| \partial_t^{n_3}u(t) \big\|_{X^{s,b}_T(\R^2)}\\
  \leq & ~ C \sum_{\substack{n_1+n_2+n_3=n,\\ \max\{ n_1,n_2,n_3 \}<n}}\| \varphi \|_{\H^{2n_1}}^{1-s}\| \varphi \|_{\H^{2n_2}}^{1-s}\| \varphi \|_{\H^{2n_3}}^{1-s}\| \varphi \|_{\H^{2n_1+1}}^s\| \varphi \|_{\H^{2n_2+1}}^s\| \varphi \|_{\H^{2n_3+1}}^s\\
  \leq & ~ C \| \varphi \|_{\H^{2n}}^{(1-s)\eta_1}\| \varphi \|_{\H^{2n}}^{(1-s)\eta_2}\| \varphi \|_{\H^{2n}}^{(1-s)\eta_3}\| \varphi \|_{\H^1}^{(1-s)(1-\eta_1)}\| \varphi \|_{\H^1}^{(1-s)(1-\eta_2)}\| \varphi \|_{\H^1}^{(1-s)(1-\eta_3)}\\
  & \times \| \varphi \|_{\H^{2n+1}}^{s\theta_1}\| \varphi \|_{\H^{2n+1}}^{s\theta_2}\| \varphi \|_{\H^{2n+1}}^{s\theta_3}\| \varphi \|_{\H^1}^{s(1-\theta_1)}\| \varphi \|_{\H^1}^{s(1-\theta_2)}\| \varphi \|_{\H^1}^{s(1-\theta_3)},
  \end{aligned}
  \end{align*}
  where $\eta_i=\frac{2n_i+1}{2n-1}$, and $\theta_i=\frac{n_i}{n}$ with $i=1,2,3$. Computing $\theta_1+\theta_2+\theta_3=1$ and $\eta_1+\eta_2+\eta_3<1$, we have
  \begin{equation}\label{eq-case 1}
  \sum_{\substack{n_1+n_2+n_3=n,\\ \max\{ n_1,n_2,n_3 \}<n}}\big\| \partial_t^{n_1}u(t) \big\|_{X^{s,b}_T(\R^2)}\big\| \partial_t^{n_2}u(t) \big\|_{X^{s,b}_T(\R^2)}\big\| \partial_t^{n_3}u(t) \big\|_{X^{s,b}_T(\R^2)} \leq C \|\varphi\|_{\H^{2n}}^{1-s}\|\varphi\|_{\H^{2n+1}}^s.
  \end{equation}

As for $II$, without loss of generality, we suppose $n_1 = 0$. Similar to the previous estimate, since $\| u_0 \|_{\H^{2n_1}} = \| \varphi \|_{L^2(\R^2)}$ is bounded, it is not necessary to introduce the parameter $\eta_1$. The conclusion is same with $I$, we omit it.

By Theorem \ref{Theorem-well-posedness}, we control $III$ with
  \begin{equation}\label{eq-case 3}
  \| u(t) \|_{X^{s,b}_T(\R^2)}^2\big\| \partial_t^{n}u(t) \big\|_{X^{s,b}_T(\R^2)} \leq \| u(t) \|_{X^{1,b}_T(\R^2)}^2\big\| \partial_t^{n}u(t) \big\|_{X^{s,b}_T(\R^2)} \leq C\| \varphi \|_{\H^1}^2\big\| \partial_t^{n}u(t) \big\|_{X^{s,b}_T(\R^2)}.
  \end{equation}

  Finally, combining \eqref{eq-partial_lu estimate}, \eqref{eq-partiall linear estimate}, \eqref{eq-case 1} and \eqref{eq-case 3}, we have
  \begin{align*}
  \big\| \partial_t^nu(t) \big\|_{X^{s,b}_T(\R^2)} \leq C \big( \|\varphi\|_{\H^{2n}}^{1-s}\|\varphi\|_{\H^{2n+1}}^s + T^\gamma\big\| \partial_t^nu(t) \big\|_{X^{s,b}_T(\R^2)} \big).
  \end{align*}
  Note that the time $T$ depends only on the $\H^1$ norm of the solution by Theorem \ref{Theorem-well-posedness}, the nonlinear part can be absorbed by the left hand side. This finishes the proof of \eqref{eq-Strichartz estimate}.

Next, we use Theorem \ref{Theorem-well-posedness} to control the left hand side of \eqref{eq-Strichartz estimate1},
\begin{equation}\label{eq-partial_lu estimate1}
\big\| \partial_t^nu(t) \big\|_{X^{s,b}_T(\R^2)} \leq C \Big( \big\| \partial_t^nu(0) \big\|_{\H^s} + T^\gamma\sum_{n_1+n_2+n_3=n}\big\| \partial_t^{n_1}u(t) \big\|_{X^{s,b}_T}\big\| \partial_t^{n_2}u(t) \big\|_{X^{s-1,b}_T}\big\| \partial_t^{n_3}u(t) \big\|_{X^{s-1,b}_T} \Big).
\end{equation}
Similar to prove the linear part and nonlinear part in \eqref{eq-Strichartz estimate}, we can obtain the estimate \eqref{eq-Strichartz estimate1}.
Therefore, we have completed the proof of the proposition.
\end{proof}

Now let $u\in C\big( (-T,T), \H^{2k+2} \big)$ be a local solution to \eqref{equation} with initial datum $\varphi\in \H^{2k+2}$. We will write $\mathcal{S}_{2k+2}$ and $\mathcal{R}_{2k+2}$ for integrals of the following form:
\begin{equation}\label{eq-S}
\mathcal{S}_{2k+2}(u)(t) = \int_{\R^2}\partial_t^kLu_0\partial_t^{m_1}Lu_1\partial_t^{m_2}u_2\partial_t^{m_3}u_3{\rm d}z, \quad m_1+m_2+m_3=k,
\end{equation}
and
\begin{equation}\label{eq-R}
\mathcal{R}_{2k+2}(u)(t) = \int_{\R^2}\partial_t^kLu_0\partial_t^{n_1}Lu_1\partial_t^{n_2}u_2\partial_t^{n_3}u_3{\rm d}z, \quad n_1+n_2+n_3=k+1, \quad n_1\leq k,
\end{equation}
where $u_0,u_1,u_2,u_3\in\{ u,\bar{u} \}$ and $L$ can be any of the following operators:
\begin{align*}
Lu = \partial_{x}u, \quad  Lu = \partial_{y}u, \quad Lu=\langle y \rangle u, \quad Lu=u.
\end{align*}

From the definition of \eqref{eq-S} and \eqref{eq-R}, we first obtain the following connection between $\H^{2k+2}$ norm with $\mathcal{S}_{2k+2}$ and $\mathcal{R}_{2k+2}$. Moreover, by using Proposition \ref{Proposition-partial_t estimate} and \ref{Proposition-strichartz estimate}, we will give a uniform bound for $\mathcal{S}_{2k+2}$ and $\mathcal{R}_{2k+2}$, which will play an important role in the proof of Theorem \ref{Theorem-growth}.

\begin{proposition}\label{Propostion-modified energy}
    Let $k\in\N, R>0$ be given. Let $T>0$ be associated with $R$ and $s_0 = 2k+2$ as in Theorem \ref{Theorem-well-posedness} and let $u\in C\big( (-T,T), \H^{2k+2} \big)$ be a local solution to \eqref{equation} with initial datum $\varphi\in \H^{2k+2}$. Then we have
\begin{equation}\label{eq-modified energy}
\frac{d}{dt}\Big( \frac12\big\| \partial_t^kAu(t,\cdot) \big\|_{L^2(\R^2)}^2 + \mathcal{S}_{2k+2}(u)(t) \Big) = \mathcal{R}_{2k+2}(u)(t).
\end{equation}
In particular, suppose that $\sup_{(-T,T)}\|u(t, \cdot)\|_{\H^1}<R$.
Then for every $\delta>0$, there exists $C>0$ such that
\begin{align}\label{eq-modefied energy S}
\sup_{t\in(-T,T)}\Big| \mathcal{S}_{2k+2}(u)(t) \Big| \leq C\|\varphi\|_{\H^{2k+2}}^{\frac{4k}{2k+1}+\delta},
\end{align}
and
\begin{align}\label{eq-modefied energy R}
\Big|\int_0^T \mathcal{R}_{2k+2}(u)(\tau){\rm d} \tau \Big| \leq C\|\varphi\|_{\H^{2k+2}}^{\frac{8k+1}{4k+2}+\delta}.
\end{align}
\end{proposition}
\begin{proof}
By the equation \eqref{equation}, we have
\begin{align*}
\partial_t^k\sqrt{A}(i\partial_tu) + \partial_t^k\sqrt{A}(Au) \pm \partial_t^k\sqrt{A}(|u|^2u) = 0.
\end{align*}
Then multiplying the equation above by $\partial_t^{k+1}\sqrt{A}\bar{u}$, we get
\begin{align*}
\frac12\frac{d}{dt}\Big( \big\| \partial_t^kAu(t,\cdot) \big\|_{L^2(\R^2)}^2 \Big) = \mp \Re\int_{\R^2}\partial_t^k\sqrt{A}(|u|^2u)\partial_t^{k+1}\sqrt{A}\bar{u}{\rm d}z.
\end{align*}
From the equation above, it suffices to compute $\Re\int_{\R^2}\partial_t^k\sqrt{A}(|u|^2u)\partial_t^{k+1}\sqrt{A}\bar{u}{\rm d}z$. Using symmetry of the operator $\sqrt{A}$ and integration by parts, we have
\begin{align}\label{eq-nonlinear part}
& \Re\int_{\R^2}\partial_t^k\sqrt{A}(|u|^2u)\partial_t^{k+1}\sqrt{A}\bar{u}{\rm d}z = \Re\int_{\R^2}\partial_t^k(|u|^2u)\partial_t^{k+1}A\bar{u}{\rm d}z\\\nonumber
= & ~ \sum_{i=1}^2\Re\int_{\R^2}\partial_t^k\partial_{x_i}(|u|^2u)\partial_t^{k+1}\partial_{x_i}\bar{u}{\rm d}z + \Re\int_{\R^2}\partial_t^k(|u|^2u)\partial_t^{k+1}(y^2\bar{u}){\rm d}z = \textup{Term~1} + \textup{Term~2},
\end{align}
here we use $\partial_{x_i}, i=1,2$ to denote  $\partial_x$ and $\partial_y$ respectively.

For $\textup{Term~1}$, we compute
\begin{equation}\label{eq-nonlinear part1}
\begin{aligned}
& \textup{Term~1}=  \sum_{i=1}^2\Big[ \Big( 2\Re\int_{\R^2}|u|^2\partial_t^k\partial_{x_i}u\partial_t^{k+1}\partial_{x_i}\bar{u}{\rm d}z + \Re\int_{\R^2} u^2\partial_t^k\partial_{x_i}\bar{u}\partial_t^{k+1}\partial_{x_i}\bar{u}{\rm d}z \Big)\\
& + \sum_{\substack{n_1+n_2+n_3=k,\\n_1<k}}\Re\Big( a_{n_1,n_2,n_3}\int_{\R^2}\partial_t^{n_1}\partial_{x_i}u\partial_t^{n_2}u\partial_t^{n_3}\bar{u}\partial_t^{k+1}\partial_{x_i}\bar{u}{\rm d}z\\
& + b_{n_1,n_2,n_3}\int_{\R^2}\partial_t^{n_1}\partial_{x_i}\bar{u}\partial_t^{n_2}u\partial_t^{n_3}u\partial_t^{k+1}\partial_{x_i}\bar{u}{\rm d}z \Big) \Big],
\end{aligned}
\end{equation}
where $a_{n_1,n_2,n_3},b_{n_1,n_2,n_3}$ are suitable real numbers. Then we can rewrite \eqref{eq-nonlinear part1} as
\begin{align*}
\begin{aligned}
& \eqref{eq-nonlinear part1} =  \frac{d}{dt}\Big[ \sum_{i=1}^2\Big( \int_{\R^2}|\partial_t^k\partial_{x_i}u|^2|u|^2{\rm d}z + \frac12\Re\int_{\R^2}\big(\partial_t^k\partial_{x_i}\bar{u}\big)^2u^2{\rm d}z \Big) \Big]\\
& - \sum_{i=1}^2\Big( \int_{\R^2}|\partial_t^k\partial_{x_i}u|^2\partial_t\big(|u|^2\big){\rm d}z + \frac12\Re\int_{\R^2}\big(\partial_t^k\partial_{x_i}\bar{u}\big)^2\partial_t\big(u^2\big){\rm d}z \Big)\\
& + \frac{d}{dt}\Big[ \sum_{i=1}^2\sum_{\substack{n_1+n_2+n_3=k,\\n_1<k}}\Re\Big( a_{n_1,n_2,n_3}\int_{\R^2}\partial_t^{n_1}\partial_{x_i}u\partial_t^{n_2}u\partial_t^{n_3}\bar{u}\partial_t^{k}\partial_{x_i}\bar{u}{\rm d}z\\
& + b_{n_1,n_2,n_3}\int_{\R^2}\partial_t^{n_1}\partial_{x_i}\bar{u}\partial_t^{n_2}u\partial_t^{n_3}u\partial_t^{k}\partial_{x_i}\bar{u}{\rm d}z \Big) \Big]\\
& - \sum_{i=1}^2\sum_{\substack{n_1+n_2+n_3=k+1,\\n_1\leq k}}\Re\Big( \tilde{a}_{n_1,n_2,n_3}\int_{\R^2}\partial_t^{n_1}\partial_{x_i}u\partial_t^{n_2}u\partial_t^{n_3}\bar{u}\partial_t^{k}\partial_{x_i}\bar{u}{\rm d}z\\
& + \tilde{b}_{n_1,n_2,n_3}\int_{\R^2}\partial_t^{n_1}\partial_{x_i}\bar{u}\partial_t^{n_2}u\partial_t^{n_3}u\partial_t^{k}\partial_{x_i}\bar{u}{\rm d}z \Big),
\end{aligned}
\end{align*}
where the last expression is a time derivative of the structure \eqref{eq-S} with $Lu = \partial_{x}u, \partial_yu$ plus the type \eqref{eq-R} with $Lu = \partial_{x}u, \partial_yu$.

As for $\textup{Term~2}$, we have
\begin{align*}
\Re\int_{\R^2}\partial_t^k(|u|^2u)\partial_t^{k+1}(y^2\bar{u}){\rm d}z = \Re\int_{\R^2}\partial_t^k\big(\langle y \rangle|u|^2u\big)\partial_t^{k+1}\big(\langle y \rangle\bar{u}\big){\rm d}z - \Re\int_{\R^2}\partial_t^k(|u|^2u)\partial_t^{k+1}\bar{u}{\rm d}z.
\end{align*}
Similar to the above estimate, and this expression will correspond to a time derivative of terms of type \eqref{eq-S} with $Lu = \langle y \rangle u$ or $Lu = u$, plus the type \eqref{eq-R} with $Lu = \langle y \rangle u$ or $Lu = u$. Therefore, this completes the proof of \eqref{eq-modified energy}.

Next, we prove \eqref{eq-modefied energy S}. From the expression \eqref{eq-S} and H\"older's inequality, we have
\begin{align}\nonumber
& \Big| \int_{\R^2}\partial_t^kLu_0\partial_t^{m_1}Lu_1\partial_t^{m_2}u_2\partial_t^{m_3}u_3\dd z \Big| \leq C \big\| \partial_t^kLu_0 \big\|_{L^2(\R^2)}\big\| \partial_t^{m_1}Lu_1 \big\|_{L^2(\R^2)}\big\| \partial_t^{m_2}u_2 \big\|_{L^\infty(\R^2)}\big\| \partial_t^{m_3}u_3 \big\|_{L^\infty(\R^2)}\\\nonumber
\leq & ~ C \big\| \partial_t^kLu_0 \big\|_{L^2(\R^2)}\big\| \partial_t^{m_1}Lu_1 \big\|_{L^2(\R^2)}\big\| \partial_t^{m_2}u \big\|_{\H^1}^{1-\delta}\big\| \partial_t^{m_2}u \big\|_{\H^2}^{\delta}\big\| \partial_t^{m_3}u \big\|_{\H^1}^{1-\delta}\big\| \partial_t^{m_3}u \big\|_{\H^2}^{\delta}\\\label{eq-s estimate}
\leq & ~ C \big\| \partial_t^{k}u \big\|_{\H^1}\big\| \partial_t^{m_1}u \big\|_{\H^1}\big\| \partial_t^{m_2}u \big\|_{\H^1}^{1-\delta}\big\| \partial_t^{m_2}u \big\|_{\H^2}^{\delta}\big\| \partial_t^{m_3}u \big\|_{\H^1}^{1-\delta}\big\| \partial_t^{m_3}u \big\|_{\H^2}^{\delta},
\end{align}
where we used the Sobolev embedding and \eqref{eq-equivalence}. Then using \eqref{eq-Strichartz estimate} and \eqref{eq-Strichartz estimate1}, we obtain
\begin{align*}
\eqref{eq-s estimate} \leq C \|\varphi\|_{\H^{2k+1}}\|\varphi\|_{\H^{2m_1+1}}\|\varphi\|_{\H^{2m_2+1}}^{1-\delta}\|\varphi\|_{\H^{2m_3+1}}^{1-\delta}
\|\varphi\|_{\H^{2m_2+2}}^\delta\|\varphi\|_{\H^{2m_3+2}}^\delta,
\end{align*}
where $\delta>0$ is small enough. Next we use interpolation and choose $\eta,\eta_1,\eta_2,\eta_3,\eta'_2,\eta'_3\in[0,1]$ such that
$
(2k+2)\eta + (1-\eta) = 2k+1, (2k+2)\eta_i + (1-\eta_i) = 2m_i+1$, with $i=1,2,3,
$
and
$
(2k+2)\eta'_i + (1-\eta'_i) = 2m_i+2$, with $i=2,3.
$
Then we obtain
\begin{align*}
\Big| \int_{\R^2}\partial_t^kLu_0\partial_t^{m_1}Lu_1\partial_t^{m_2}u_2\partial_t^{m_3}u_3\dd z \Big| \leq C \|\varphi\|_{\H^{2k+2}}^{\eta+\eta_1+(1-\delta)(\eta_2+\eta_3)+\delta(\eta'_2+\eta'_3)},
\end{align*}
and we conclude the proof of \eqref{eq-modefied energy S} by $\eta+\eta_1+(1-\delta)(\eta_2+\eta_3)+\delta(\eta'_2+\eta'_3)=\frac{4k}{2k+1}+\frac{2\delta}{2k+1}$.

Finally, we just need to estimate \eqref{eq-modefied energy R}. From the formula \eqref{eq-R}, we may assume $n_2\geq1$(same with $n_3\geq1$) and obtain
\begin{align*}
\begin{aligned}
& \Big| \int_0^T\int_{\R^2}\partial_t^kLu_0\partial_t^{n_1}Lu_1\partial_t^{n_2}u_2\partial_t^{n_3}u_3{\rm d}z{\rm d}\tau \Big| \leq \Big| \int_0^T\int_{\R^2}\big(\partial_t^kLu_0\big)\big(\partial_t^{n_1}Lu_1\big)\big(\partial_t^{n_2-1}Au_2\big)\big(\partial_t^{n_3}u_3\big){\rm d}z{\rm d}\tau \Big|\\
& + \int_0^T\int_{\R^2}\big|\partial_t^kLu\big|\big|\partial_t^{n_1}Lu\big|\big|\partial_t^{n_2-1}(|u|^2u)\big|\big|\partial_t^{n_3}u_3\big|{\rm d}z{\rm d}\tau
\triangleq I + II,
\end{aligned}
\end{align*}
where $n_1,n_2,n_3$ satisfy $n_1+n_2+n_2=k+1$ with $n_1\leq k$.

We first consider the term $II$. By the Cauchy-Schwarz inequality, Sobolev embedding and Proposition \ref{Proposition-strichartz estimate}, we get
\begin{align*}
\begin{aligned}
II \leq & ~ \int_0^T\big\| \partial_t^kLu(\tau) \big\|_{L^2(\R^2)}\big\| \partial_t^{n_1}Lu(\tau) \big\|_{L^2(\R^2)}\dd \tau\big\| \partial_t^{n_2-1}(|u|^2u) \big\|_{L^\infty((0,T);\H^{1+\delta})}\big\| \partial_t^{n_3}u \big\|_{L^\infty((0,T);\H^{1+\delta})}\\
\leq & ~ C \|\varphi\|_{\H^{2k+1}}\|\varphi\|_{\H^{2n_1+1}}\|\varphi\|_{\H^{2n_3+1}}^{1-\delta}\|\varphi\|_{\H^{2n_3+2}}^\delta\big\| \partial_t^{n_2-1}(|u|^2u) \big\|_{L^\infty((0,T);\H^{1+\delta})}\\
\leq & ~ C \|\varphi\|_{\H^{2k+2}}^{\theta+\theta_1+(1-\delta)\theta_3}\|\varphi\|_{\H^1}^{3-(\theta+\theta_1+(1-\delta)\theta_3+\delta\theta')}\big\| \partial_t^{n_2-1}(|u|^2u) \big\|_{L^\infty((0,T);\H^{1+\delta})},
\end{aligned}
\end{align*}
where $(2k+2)\theta + (1-\theta) = 2k+1, (2k+2)\theta_i + (1-\theta_i) = 2n_i+1$, with $i=1,3$, and $(2k+2)\theta' + (1-\theta') = 2n_3+2$. Therefore, we obtain
\begin{equation}\label{eq-r-ii}
II \leq C \|\varphi\|_{\H^{2k+2}}^{\frac{2k+2n_1+2n_3}{2k+1}}\|\varphi\|_{\H^{2k+2}}^\delta\big\| \partial_t^{n_2-1}(|u|^2u) \big\|_{L^\infty((0,T);\H^{1+\delta})},
\end{equation}
and it suffices to estimate $\big\| \partial_t^{n_2-1}(|u|^2u) \big\|_{L^\infty((0,T);\H^{1+\delta})}$. In fact, using Leibniz rule and the fact that $\H^{1+\delta}$ is an algebra, we have
\begin{align*}
\begin{aligned}
& \big\| \partial_t^{n_2-1}(|u|^2u) \big\|_{L^\infty((0,T);\H^{1+\delta})}\\
\leq & ~ C\sum_{j_1+j_2+j_3=n_2-1}\big\| \partial_t^{j_1}u \big\|_{L^\infty((0,T);\H^{1+\delta})}\big\| \partial_t^{j_2}u \big\|_{L^\infty((0,T);\H^{1+\delta})}\big\| \partial_t^{j_3}u \big\|_{L^\infty((0,T);\H^{1+\delta})}\\
\leq & ~ C\sum_{j_1+j_2+j_3=n_2-1} \|\varphi\|_{\H^{2j_1+1}}^{1-\delta}\|\varphi\|_{\H^{2j_2+1}}^{1-\delta}
\|\varphi\|_{\H^{2j_3+1}}^{1-\delta}\|\varphi\|_{\H^{2k+2}}^{3\delta}\\
\leq & ~ C\|\varphi\|_{\H^{2k+2}}^{(1-\delta)(\eta_1+\eta_2+\eta_3)}
\|\varphi\|_{\H^1}^{(1-\delta)(3-\eta_1-\eta_2-\eta_3)}\|\varphi\|_{\H^{2k+2}}^{3\delta},
\end{aligned}
\end{align*}
where $(2k+2)\eta_i + (1-\eta_i) = 2j_i+1$, with $i=1,2,3$. Then we have
\begin{equation}\label{eq-r-ii1}
\big\| \partial_t^{n_2-1}(|u|^2u) \big\|_{L^\infty((0,T);\H^{1+\delta})} \leq C\|\varphi\|_{\H^{2k+2}}^{\frac{2n_2-2}{2k+1}+3\delta}.
\end{equation}
Therefore, combining \eqref{eq-r-ii} and \eqref{eq-r-ii1}, we get
\begin{align*}
II \leq C \|\varphi\|_{\H^{2k+2}}^{\frac{2k+2n_1+2n_3+2n_2-2}{2k+1}+4\delta} \leq C \|\varphi\|_{\H^{2k+2}}^{\frac{4k}{2k+1}+\delta}.
\end{align*}

Now we focus on estimating $I$, by Littlewood-Paley decomposition, and depending on the frequencies $N_0,N_1,N_2,N_3$, we will split the sum in serval cases
\begin{align*}
I \leq & \sum_{N_0,N_1,N_2,N_3} \Big| \int_0^T\int_{\R^2} \Delta_{N_0}\big(\partial_t^kLu_0\big)\Delta_{N_1}\big(\partial_t^{n_1}Lu_1\big)\Delta_{N_2}\big(\partial_t^{n_2-1}Au_2\big)
\Delta_{N_3}\big(\partial_t^{n_3}u_3\big){\rm d}z{\rm d}\tau \Big|\\
\leq & \sum_{\min\{N_0,N_2\}\geq \max\{N_1,N_3\}} \Big| \int_0^T\int_{\R^2} \Delta_{N_0}\big(\partial_t^kLu_0\big)\Delta_{N_1}\big(\partial_t^{n_1}Lu_1\big)\Delta_{N_2}\big(\partial_t^{n_2-1}Au_2\big)
\Delta_{N_3}\big(\partial_t^{n_3}u_3\big){\rm d}z{\rm d}\tau \Big|\\
 & + \sum_{\min\{N_0,N_1\}\geq \max\{N_2,N_3\}} \Big| \int_0^T\int_{\R^2} \Delta_{N_0}\big(\partial_t^kLu_0\big)\Delta_{N_1}\big(\partial_t^{n_1}Lu_1\big)\Delta_{N_2}\big(\partial_t^{n_2-1}Au_2\big)
\Delta_{N_3}\big(\partial_t^{n_3}u_3\big){\rm d}z{\rm d}\tau \Big|\\
 & + \sum_{\min\{N_0,N_3\}\geq \max\{N_1,N_2\}} \Big| \int_0^T\int_{\R^2} \Delta_{N_0}\big(\partial_t^kLu_0\big)\Delta_{N_1}\big(\partial_t^{n_1}Lu_1\big)\Delta_{N_2}\big(\partial_t^{n_2-1}Au_2\big)
\Delta_{N_3}\big(\partial_t^{n_3}u_3\big){\rm d}z{\rm d}\tau \Big|\\
 & + \sum_{\min\{N_1,N_2\}\geq \max\{N_0,N_3\}} \Big| \int_0^T\int_{\R^2} \Delta_{N_0}\big(\partial_t^kLu_0\big)\Delta_{N_1}\big(\partial_t^{n_1}Lu_1\big)\Delta_{N_2}\big(\partial_t^{n_2-1}Au_2\big)
\Delta_{N_3}\big(\partial_t^{n_3}u_3\big){\rm d}z{\rm d}\tau \Big|\\
 & + \sum_{\min\{N_1,N_3\}\geq \max\{N_0,N_2\}} \Big| \int_0^T\int_{\R^2} \Delta_{N_0}\big(\partial_t^kLu_0\big)\Delta_{N_1}\big(\partial_t^{n_1}Lu_1\big)\Delta_{N_2}\big(\partial_t^{n_2-1}Au_2\big)
\Delta_{N_3}\big(\partial_t^{n_3}u_3\big){\rm d}z{\rm d}\tau \Big|\\
 & + \sum_{\min\{N_2,N_3\}\geq \max\{N_0,N_1\}} \Big| \int_0^T\int_{\R^2} \Delta_{N_0}\big(\partial_t^kLu_0\big)\Delta_{N_1}\big(\partial_t^{n_1}Lu_1\big)\Delta_{N_2}\big(\partial_t^{n_2-1}Au_2\big)
\Delta_{N_3}\big(\partial_t^{n_3}u_3\big){\rm d}z{\rm d}\tau \Big|\\
=& ~ \mathcal{A} + \mathcal{B} + \mathcal{C} + \mathcal{D} + \mathcal{E} + \mathcal{F}.
\end{align*}

We apply the Cauchy--Schwarz inequality and Theorem \ref{Theorem-linear bilinear estimate} to estimate the case $\mathcal{A}$ by
\begin{align*}
\begin{aligned}
& \Big| \int_0^T\int_{\R^2} \Delta_{N_0}\big(\partial_t^kLu_0\big)\Delta_{N_1}\big(\partial_t^{n_1}Lu_1\big)\Delta_{N_2}\big(\partial_t^{n_2-1}Au_2\big)
\Delta_{N_3}\big(\partial_t^{n_3}u_3\big){\rm d}z{\rm d}\tau \Big|\\
\leq & ~ \big\| \Delta_{N_0}\big(\partial_t^kLu_0\big)\Delta_{N_1}\big(\partial_t^{n_1}Lu_1\big) \big\|_{L^2((0,T);L^2(\R^2))}\big\| \Delta_{N_2}\big(\partial_t^{n_2-1}Au_2\big)\Delta_{N_3}\big(\partial_t^{n_3}u_3\big) \big\|_{L^2((0,T);L^2(\R^2))}\\
\leq & ~ C \Big(\frac{N_1N_3}{N_0N_2}\Big)^{\frac12-\delta}\!\big\| \Delta_{N_0}\big(\partial_t^kLu_0\big) \big\|_{X^{0,b}_T}\big\| \Delta_{N_1}\big(\partial_t^{n_1}Lu_1\big) \big\|_{X^{0,b}_T}\big\| \Delta_{N_2}\big(\partial_t^{n_2-1}Au_2\big) \big\|_{X^{0,b}_T}\big\| \Delta_{N_3}\big(\partial_t^{n_3}u_3\big) \big\|_{X^{0,b}_T}\\
\leq & ~ C \Big(\frac{N_3}{N_2}\Big)^{\frac12-\delta}\big\| \Delta_{N_0}\big(\partial_t^kLu_0\big) \big\|_{X^{0,b}_T}\big\| \Delta_{N_1}\big(\partial_t^{n_1}Lu_1\big) \big\|_{X^{0,b}_T}\big\| \Delta_{N_2}\big(\partial_t^{n_2-1}Au_2\big) \big\|_{X^{0,b}_T}\big\| \Delta_{N_3}\big(\partial_t^{n_3}u_3\big) \big\|_{X^{0,b}_T}\\
\leq & ~ C \big\| \Delta_{N_0}\big(\partial_t^kLu_0\big) \big\|_{X^{0,b}_T}\big\| \Delta_{N_1}\big(\partial_t^{n_1}Lu_1\big) \big\|_{X^{0,b}_T}\big\| \Delta_{N_2}\big(\partial_t^{n_2-1}Au_2\big) \big\|_{X^{-\frac12+\delta,b}_T}\big\| \Delta_{N_3}\big(\partial_t^{n_3}u_3\big) \big\|_{X^{\frac12-\delta,b}_T}.
\end{aligned}
\end{align*}
Summarizing and using Lemma \ref{Proposition-bougain space estimate}, we  estimate $\mathcal{A}$ as
\begin{equation}
\begin{aligned}\label{eq-Rcase1}
\mathcal{A} \leq & ~ C \big\| \partial_t^kLu \big\|_{X^{\delta,b}_T}\big\| \partial_t^{n_1}Lu \big\|_{X^{\delta,b}_T}\big\| \partial_t^{n_2-1}Au \big\|_{X^{-\frac12+\delta,b}_T}\big\| \partial_t^{n_3}u \big\|_{X^{\frac12-\delta,b}_T}\\
\leq & ~ C \big\| \partial_t^ku \big\|_{X^{1+\delta,b}_T}\big\| \partial_t^{n_1}u \big\|_{X^{1+\delta,b}_T}\big\| \partial_t^{n_2-1}u \big\|_{X^{\frac32+\delta,b}_T}\big\| \partial_t^{n_3}u \big\|_{X^{\frac12-\delta,b}_T}
\end{aligned}
\end{equation}
for $b>\frac12$.

In order to reduce to the case $\mathcal{A}$, we use the inequality $N_2^2 \leq N_0N_1$ so that $\frac{N_2N_3}{N_0N_1}\leq\frac{N_3}{N_2}$ for $\mathcal{B}$. Similarly for $\mathcal{C}$, $\mathcal{D}$, $\mathcal{E}$, and $\mathcal{F}$, we use the inequality $N_1N_2^2 \leq N_0N_3^2$ so that $\frac{N_1N_2}{N_0N_3}\leq\frac{N_3}{N_2}$, $N_0 \leq N_1$ such that $\frac{N_0N_3}{N_1N_2}\leq\frac{N_3}{N_2}$, $N_0N_2^2 \leq N_1N_3^2$ so that $\frac{N_0N_2}{N_1N_3}\leq\frac{N_3}{N_2}$, and $N_0N_1 \leq N_3^2$ so that $\frac{N_0N_1}{N_2N_3}\leq\frac{N_3}{N_2}$ respectively. Hence, we have to estimate \eqref{eq-Rcase1}. By \eqref{eq-Strichartz estimate} and \eqref{eq-Strichartz estimate1} in Proposition \ref{Proposition-strichartz estimate},
\begin{equation}
\begin{aligned}\label{eq-Rcase1'}
\eqref{eq-Rcase1} \leq  \|\varphi\|_{\H^{2k+1}}^{1-\delta}\|\varphi\|_{\H^{2n_1+1}}^{1-\delta}
\|\varphi\|_{\H^{2n_2-1}}^{\frac12-\delta}\|\varphi\|_{\H^{2n_2}}
^{\frac12+\delta}\|\varphi\|_{\H^{2n_3}}^{\frac12-\delta}\|\varphi\|_{\H^{2n_3+1}}^{\frac12+\delta}
\|\varphi\|_{\H^{2n_1+2}}
^{\delta}\|\varphi\|_{\H^{2k+2}}^\delta.
\end{aligned}
\end{equation}
By interpolation and the \emph{a priori} bound on the $\H^1$ norm, we have
\begin{align*}
\begin{aligned}
\eqref{eq-Rcase1'}
\leq  C \|\varphi\|_{\H^{2k+2}}^{(1-\delta)\theta+(1-\delta)\theta_1+(\frac12-\delta)\gamma_2+(\frac12+\delta)\eta_2+(\frac12-\delta)\eta_3
+(\frac12+\delta)\theta_3+\delta\alpha_1+\delta},
\end{aligned}
\end{align*}
where $C$ contains some power of $\|\varphi\|_{\H^{1}}$, $\theta=\frac{2k}{2k+1}, \alpha_1=\frac{2n_1+1}{2k+1}, \gamma_2=\frac{2n_2-2}{2k+1}$, $\theta_i=\frac{2n_i}{2k+1}$ with $i=1,3$, and $\eta_j=\frac{2n_j-1}{2k+1}$ with $j=2,3$.

We need to deal  with $n_3\geq1$ and $n_3=0$ differently: for $n_3\geq1$, we obtain
$$
(1-\delta)\theta+(1-\delta)\theta_1+\left(\frac12-\delta\right)\gamma_2+\left(\frac12+\delta\right)\eta_2+\left(\frac12-\delta\right)\eta_3
+\left(\frac12+\delta\right)\theta_3+\delta\alpha_1 +\delta = \frac{4k}{2k+1} + \frac{4\delta}{2k+1},
$$
while for $n_3=0$,
$$
(1-\delta)\theta+(1-\delta)\theta_1+(\frac12-\delta)\gamma_2+(\frac12+\delta)\eta_2
+\delta\alpha_1 +\delta = \frac{8k+1}{4k+2} + \frac{3\delta}{2k+1}.
$$
This completes the proof of \eqref{eq-modefied energy R}, after relabeling $\frac{3\delta}{2k+1}$ to $\delta$.
\end{proof}

Finally, by using suitable energies and estimates, we prove the Sobolev norms growth provided in Theorem \ref{Theorem-growth}.

\begin{proof}[Proof of Theorem \ref{Theorem-growth}]

By integrating the identity \eqref{eq-modified energy} on the strip $(0,T)$  in Proposition \ref{Propostion-modified energy}, we have
\begin{align*}
\frac12\big\| \partial_t^kAu(T,\cdot) \big\|_{L^2(\R^2)}^2 + \mathcal{S}_{2k+2}(u)(T) = \frac12\big\| \partial_t^kAu(0,\cdot) \big\|_{L^2(\R^2)}^2 + \mathcal{S}_{2k+2}(u)(0) + \int_0^T\mathcal{R}_{2k+2}(u)(\tau)\dd \tau,
\end{align*}
where $T$ is the one defined in Proposition \ref{Propostion-modified energy} and $R$ is defined as
\begin{equation}\label{eq-A11}
\sup_{t\in\mathbb R}\|u(t, \cdot)\|_{\H^1} = R,
\end{equation}
and $R<\infty$ by \eqref{eq-A1} in Theorem \ref{Theorem-growth}. Then we use \eqref{eq-modefied energy S} and \eqref{eq-modefied energy R} to obtain
\begin{align*}
\frac12\big\| \partial_t^kAu(T,\cdot) \big\|_{L^2(\R^2)}^2 - \frac12\big\| \partial_t^kAu(0,\cdot) \big\|_{L^2(\R^2)}^2 \leq C\|u(0,\cdot)\|_{\H^{2k+2}}^{\frac{8k+1}{4k+2}+\delta}.
\end{align*}

Next by \eqref{eq-A11}, we can iterate the bound above with the same constants,
\begin{align*}
\frac12\big\| \partial_t^kAu\big((n+1)T,\cdot\big) \big\|_{L^2(\R^2)}^2 - \frac12\big\| \partial_t^kAu(nT,\cdot) \big\|_{L^2(\R^2)}^2 \leq C\|u(nT,\cdot)\|_{\H^{2k+2}}^{\frac{8k+1}{4k+2}+\delta}
\end{align*}
for all $t\in\big(nT,(n+1)T\big)$ with $n\in\N$. By summing up for $n\in[0,N-1]$, we have
\begin{align*}
\big\| \partial_t^kAu(NT,\cdot) \big\|_{L^2(\R^2)}^2 \leq C \sum_{n\in\{ 0,...,N-1 \}}\|u(nT,\cdot)\|_{\H^{2k+2}}^{\frac{8k+1}{4k+2}+\delta},
\end{align*}
and it implies
\begin{align*}
\sup_{n\in[0,N]}\big\| \partial_t^kAu(nT,\cdot) \big\|_{L^2(\R^2)}^2 \leq CN \Big(\sup_{n\in[0,N]}\|u(nT,\cdot)\|_{\H^{2k+2}}\Big)^{\frac{8k+1}{4k+2}+\delta}.
\end{align*}
Then using Proposition \ref{Proposition-partial_t estimate}, we have
\begin{align*}
\sup_{n\in[0,N]}\| u(nT,\cdot) \|_{\H^{2k+2}} \leq CN^{\frac{2}{3}(2k+1)+\delta}.
\end{align*}
In particular,
\begin{align*}
\| u(NT,\cdot) \|_{\H^{2k+2}} \leq CN^{\frac{2}{3}(2k+1)+\delta},
\end{align*}
for all $ N\in\N$. Using Theorem \ref{Theorem-well-posedness}, it is easy to deduce that the estimate above implies
\begin{align*}
\sup_{t\in[NT,(N+1)T]}\| u(t,\cdot) \|_{\H^{2k+2}} \leq CN^{\frac{2}{3}(2k+1)+\delta}
\end{align*}
provided that we suitably modify the constant $C$. Summarizing, we get
\begin{align*}
\| u(t,\cdot) \|_{\H^{2k+2}} \leq Ct^{\frac{2}{3}(2k+1)+\delta}
\end{align*}
for all $t>0$. This finishes the proof of Theorem 1.
\end{proof}

\end{document}